\documentclass[twoside,11pt,reqno]{amsart}
\usepackage{amsmath,amssymb,amscd}
\usepackage[small,nohug,heads=vee]{diagrams}
\diagramstyle[labelstyle=\scriptstyle]

\makeatletter

\iftrue
\@settopoint\oddsidemargin
\@settopoint\marginparwidth
\@settopoint\evensidemargin
\@settopoint\topmargin
\fi

\newtheorem{Proposition}{Proposition}[section]

\newtheorem{Lemma}[Proposition]{Lemma}
\newtheorem{Theorem}[Proposition]{Theorem}
\newtheorem{Corollary}[Proposition]{Corollary}

\theoremstyle{definition}
  \newtheorem{Remark}[Proposition]{Remark}

\@addtoreset{equation}{section}

\def\C{{\mathbb C}}

\def\Z{{\mathbb Z}}

\def\eps{\epsilon}
\def\sgn{{\operatorname{sgn}}}

\def\rdet{\operatorname{rdet}}
\def\row{\operatorname{row}}
\def\col{\operatorname{col}}
\def\gr{\operatorname{gr}}
\def\ev{\operatorname{ev}}
\def\pr{\operatorname{pr}}
\def\trace{\operatorname{tr}}
\def\Ann{{\operatorname{Ann}}}

\def\RA{{\operatorname{RA}}}
\def\RB{{\operatorname{RB}}}
\def\CL{{\operatorname{CL}}}
\def\CR{{\operatorname{CR}}}
\def\Tab{{\operatorname{Tab}}}
\def\Row{{\operatorname{Row}}}
\def\Col{{\operatorname{Col}}}
\def\Std{{\operatorname{Std}}}

\def\ad{{\operatorname{ad}}}
\def\Prim{{\operatorname{Prim}}}

\newdimen\hoogte    \hoogte=14pt    
\newdimen\breedte   \breedte=15pt   
\newdimen\dikte     \dikte=0.5pt    

\newenvironment{young}{\begingroup
       \def\vr{\vrule height0.8\hoogte width\dikte depth 0.2\hoogte}
       \def\fbox##1{\vbox{\offinterlineskip
                    \hrule height\dikte
                    \hbox to \breedte{\vr\hfill##1\hfill\vr}
                    \hrule height\dikte}}
       \vbox\bgroup \offinterlineskip \tabskip=-\dikte \lineskip=-\dikte
            \halign\bgroup &\fbox{##\unskip}\unskip  \crcr }
       {\egroup\egroup\endgroup}
\def\ydiagram#1{\relax\ifmmode\vcenter{\,\begin{young}#1\end{young}\,}\else%
              $\vcenter{\,\begin{young}#1\end{young}\,}$\fi}

\def\mf {\mathfrak}

\newcommand{\set}[2]{\ensuremath{\{#1\:|\:#2\}}}

\begin{document}

\title[Representation theory of rectangular finite $W$-algebras]
{\boldmath Representation theory of rectangular finite $W$-algebras}
\author{Jonathan Brown}
\address{School of Mathematics, University of Birmingham, Edgbaston, Birmingham, B15 2TT, UK}
\email{brownjs@maths.bham.ac.uk}

\begin{abstract}
We classify the finite dimensional irreducible representations of
rectangular finite $W$-algebras, i.e., the finite 
$W$-algebras $U(\mf{g}, e)$ where
$\mf{g}$ is a symplectic or orthogonal Lie algebra and $e \in \mf{g}$ 
is a nilpotent element with
Jordan blocks all the same size.  
\end{abstract}

\maketitle

\section{Introduction}

This paper concerns the representation theory
of the finite $W$-algebra $U(\mathfrak{g}, e)$ associated to a nilpotent
element $e$ in a reductive Lie algebra $\mathfrak{g}$. 
The main focus of this paper is the
representation theory of the finite
$W$-algebras associated to nilpotent elements in the 
symplectic or orthogonal Lie algebras whose Jordan blocks are all the same size.
We refer to these simply as {\em rectangular} finite $W$-algebras.

The general definition of finite $W$-algebras
is due to Premet in \cite{P1}, though in some cases they
had been introduced much earlier by Lynch in \cite{Ly} following Kostant's
celebrated work on Whittaker modules in \cite{Ko}.
The terminology ``finite $W$-algebra'' comes from the mathematical
physics literature,
where finite $W$-algebras are the finite type
analogs of the vertex $W$-algebras
defined and studied for example by Kac, Roan, and Wakimoto in \cite{KRW}.
The precise identification between the definitions
in \cite{P1} and \cite{KRW}
was made only recently by D'Andrea, De Concini, De Sole, Heluani, and Kac  in \cite{DDDHK}.

There are many remarkable connections between finite $W$-algebras and other areas of mathematics.
The finite $W$-algebra $U(\mf{g},e)$ possesses two natural filtrations,
the {\em Kazhdan} and {\em loop} filtrations.
The main structure theorem for finite $W$-algebras, proved in \cite{P1} and reproved
in \cite{GG}, is
that the associated graded algebra to $U(\mf{g},e)$ with respect to
the Kazhdan
filtration is isomorphic to the coordinate algebra
of the Slodowy slice, i.e. $U(\mf{g},e)$ is a {\em quantization} of the Slodowy
slice through the nilpotent
orbit containing $e$.
On the other hand, by \cite{P2} the
associated graded algebra with respect to
the loop filtration is isomorphic to $U(\mf{g}^e)$, the universal enveloping algebra of
the centralizer of $e$ in $\mf{g}$.
Because of this, the structure of $U(\mathfrak{g},e)$ is intimately
related to the invariant theory of the centralizer $\mf{g}^e$.
In \cite{BB} this connection
was used to construct a system of
algebraically independent generators for the center of the universal
enveloping algebra $U(\mf{g}^e)$
in the case $\mf{g}=\mf{gl}_n(\C)$,
giving a constructive proof of the freeness of
this center (which
had been established earlier by Panyushev, Premet and Yakimova in \cite{PPY}
by a different method) and also verifying \cite[Conjecture 4.1]{PPY}.

The work of Premet in \cite{P2,P3}, Losev in \cite{Lo1,Lo2},
and Ginzburg in \cite{Gi} has
highlighted the importance of the study of finite dimensional
representations of $U(\mf{g},e)$, revealing an intimate relationship
with the theory of primitive ideals of the
universal enveloping algebra $U(\mf{g})$ itself.
At the heart of this connection is an equivalence of
categories due to Skryabin in \cite{Skryabin} between the category of $U(\mf{g},e)$-modules
and a certain category of generalized Whittaker modules for $\mf{g}$.
For other recent results about the representation theory of finite $W$-algebras
see e.g. \cite{Lo3}, \cite{Lo4}, \cite{Go}, \cite{GRU}.

\subsection{Statement of the main results}

Throughout this paper we denote the general linear, 
symplectic, and orthogonal Lie algebras $\mf{gl}_n(\C)$,
$\mf{sp}_n(\C)$, and $\mf{so}_n(\C)$ as 
$\mf{g}_n$, $\mf{g}_n^-$, and $\mf{g}_n^-$ for short, 
assuming that $n$ is even if $\mf{g} = \mf{sp}_n(\C)$.
We will also need the following index set defined in terms of 
a positive integer $n$:
\[
  \mathcal{I}_n = \{ 1-n, 3-n, \dots, n-1 \}.
\]
Let $Y_n^+$ and  $Y_n^-$ denote the twisted Yangians associated to
$\mf{g}_n^+$ and $\mf{g}_n^-$, respectively.
These are certain associative algebras with generators 
$\set{S_{i,j}^{(r)}} { i, j \in \mathcal{I}_n, r \in \Z_{>0}}$;
see
 \cite{MNO}
 for the full relations.
Fix positive integers $n$ and $l$, and a sign $\eps \in \{\pm\}$,
now let $\mf{g} = \mf{g}_{nl}^\eps$.
Let $e$ be a nilpotent element of Jordan type $(l^n)$ in
$\mf{g}$.
In order to ensure that such a nilpotent exists one must further assume that
if $\eps = +$ and $l$ is even then $n$ is even, and that
if $\eps = -$ and $l$ is odd then $n$ is even.
Let $U(\mf{g},e)$ be the finite $W$-algebra attached to $\mf{g}$ and the nilpotent
element $e$; see $\S$\ref{section2} below for the general definition.
We will also need another sign $\phi$
defined to be $\eps$ if $l$ is odd, and $-\eps$ if $l$ is even.
Set $Y=Y_n^\phi$.
The main result of \cite{Br}
is the following theorem:
\begin{Theorem} \label{th_ch3}
There exists a surjective algebra homomorphism
  $Y \twoheadrightarrow U(\mf{g},e)$
with kernel generated by the elements
\begin{equation} \label{ker2}
\begin{array}{rl}
  \left\{S_{i,j}^{(r)} \:\Big|\: i,j \in \mathcal{I}_n, r > l\right\}_{\phantom{S}}
    & \text{if $l$ is even;} \\  \\
\left\{ S_{i,j}^{(r)} +\frac{\phi}{2}S_{i,j}^{(r-1)}\: \Big| \: i,j \in \mathcal{I}_n, 
       r > l\right\}_{\phantom{S}} & \text{if $l$ is odd.}
\end{array}
\end{equation}
\end{Theorem}

Results along these lines were first noticed by Ragoucy in \cite{R}, where
he observed that a similar homomorphism exists in the case that $l$ is odd
for certain commutative analogs of these algebras.

The main aim of the present  
article is to combine this theorem with Molev's classification
of the finite dimensional
irreducible representations of twisted Yangians
from \cite{M}
to deduce a classification
of finite dimensional irreducible representations of the
rectangular finite $W$-algebras.
The main combinatorial objects in this classification are
skew-symmetric $n \times l$ tableaux.
A {\em skew-symmetric $n \times l$ tableaux}
is an
 $n \times l$ matrix of complex
numbers, with rows labeled in order from top to bottom by the set
$\mathcal{I}_n$ and columns labeled in order from left to right by the set
$\mathcal{I}_l$,  and which is skew-symmetric with respect to the center of the
matrix,
that is, if $A=(a_{i,j})_{i \in \mathcal{I}_n, j \in \mathcal{I}_l}$
is a skew-symmetric $n \times l$ tableaux then $a_{i,j} = -a_{-i,-j}$. 
Let $\Tab_{n,l}$ denote the set of skew-symmetric $n \times l$ tableaux. 
We say that two skew-symmetric $n \times l$ tableaux are 
{\em row equivalent} if one can be obtained
from the other by permuting entries within rows.
Let $\Row_{n,l}$ denote the set of row equivalence classes of 
skew-symmetric $n \times l$ tableaux.  

In the following definition (and from here on) we use the 
partial order $\geq$ on $\C$ defined by $a \geq b
$ if $a-b \in \Z_{\geq 0}$.
A skew-symmetric $n \times l$ tableaux
$A = (a_{i,j})_{i \in \mathcal{I}_n, j \in \mathcal{I}_l}$ 
is {\em $\eps$-column strict} if 
\begin{itemize}
\item[-] the entries in every column except for the middle column
  (which exists only when $l$ is odd) are strictly 
  decreasing from top to bottom, i.e., 
   $a_{1-n,j} > a_{3-n, j} > \dots > a_{n-1,j}$ 
   for all $0 \ne j \in \mathcal{I}_l$;
\item[-] if $l$ is odd and $n$ is even 
then the entries in the middle column satisfy 
$a_{1-n,0} > a_{3-n,0} > \cdots > a_{-1,0}$, 
and they also satisfy $a_{-1,0} > 0$ if
$\eps=-$, and they satisfy 
$a_{-3,0} + a_{-1,0} > 0$ 
if $\eps=+$ and $n \ge 4$;
\item[-] if $l$ is odd and $n$ is odd then the entries in
the middle column satisfy
$a_{1-n,0} > a_{3-n,0} > \cdots > a_{-2,0}$, 
and they also satisfy $2 a_{-2,0} > 0$.
\end{itemize}
Let $\Col_{n,l}^\eps$ denote the set of all $\eps$-column strict skew-symmetric
$n \times l$ tableaux, and
let $\Std_{n,l}^\eps$ denote the set of 
elements of $\Row_{n,l}$
which have a representative
in $\Col^\eps_{n,l}$.  

We relate these sets to certain representations of the twisted Yangian $Y$.
It is convenient to use the power series
\begin{equation}
  S_{i,j} (u) = \sum_{r \geq 0} S_{i,j}^{(r)} u^{-r} \in Y[[u^{-1}]],
\end{equation}
where $S_{i,j}^{(0)} = \delta_{i,j}$.
A $Y$-module $V$ is called a {\em highest weight module} if
it generated by a vector $v$ such that $S_{i,j}(u) v = 0$ for all $i
< j$ and 
if for all $i$ we have that 
$S_{i,i}(u) v = \mu_i(u) v$ 
for some power series 
$\mu_i(u) \in 1 + u^{-1}\C[[u^{-1}]]$. 
To a skew-symmetric $n \times l$
tableaux $A = (a_{i,j})_{i \in I_n, j \in I_l}$ we associate a
(unique up to isomorphism) irreducible
highest weight $Y$-module 
generated by a highest weight  
vector $v$
for which
\[
  (u - \frac{i}{2})^l S_{i,i}(u-\frac{i}{2}) v = (u+a_{i,1-l})(u+a_{i,3-l}) \dots (u+a_{i,l-1}) v
\]
 if $l$ is even and $i \ge 0$, or
\begin{align*}
  (&u - \frac{i}{2})^{l-1} (u+\frac{\phi-i}{2}) S_{i,i}(u-\frac{i}{2}) v  \\
   &= 
   (u+a_{i,1-l})(u+a_{i,3-l}) \dots (u+a_{i,-2})(u+a_{i,0} + 
\delta_{i,0} /2) (u+a_{i,2}) \dots (u+a_{i,l-1}) v 
\end{align*}
if $l$ is odd and $i \ge 0$.
This $Y$-module factors through the
surjection $Y \twoheadrightarrow U(\mf{g},e)$  from
Theorem \ref{th_ch3} to yield a
(not necessarily finite dimensional)
irreducible $U(\mf{g},e)$-module denoted $L(A)$ for each $A \in \Row_{n,l}$. 
Moreover these are the only highest weight $Y$-modules which descend to 
$U(\mf{g},e)$,
so the
problem of classifying the finite dimensional irreducible  
representations is reduced to determining exactly which $L(A)$'s are
finite dimensional, which can be deduced from Molev's results in \cite{M}. 
The following is the main theorem of this paper:
\begin{Theorem} \label{t_class}
Suppose $A \in \Row_{n,l}$.
\begin{enumerate}
\item
If $l$ is odd or if $l$ is even and $\eps=+$
then 
$L(A)$ is finite dimensional if and only if 
$A$ has a representative in $\Col_{n,l}^\eps$.
Hence
\[
  \set{ L(A)} { A \in \Std_{n,l}^\eps}
\]
is a complete set of isomorphism classes of 
the finite dimensional irreducible representations of 
$U(\mf{g},e)$.

\item
If $l$ is even and $\eps = -$ then 
$L(A)$ is finite dimensional if and only if 
$A^+$ has a representative in $\Col_{n,l+1}^+$.
Hence
\[
  \set{ L(A) }{ A \in \Row_{n,l}, 
    A^+ \in \Std^+_{n,l+1} }
\]
is a complete set of isomorphism classes of 
the finite dimensional
irreducible representations of $U(\mf{g},e)$.
\end{enumerate}
\end{Theorem}
In the theorem
$A^+$ denotes the skew-symmetric $n \times (l+1)$ tableaux
obtained by inserting a middle column into $A$ with entries
\[
  \frac{n}{2} - 1, \frac{n}{2} -2, \dots, 1, 0,0, -1, -2, \dots, 1-\frac{n}{2}
\] 
if $n$ is even and 
\[
  \frac{n}{2} - 1, \frac{n}{2} -2, \dots, \frac{1}{2}, 0,
   -\frac{1}{2}, -\frac{3}{2},  \dots, 1-\frac{n}{2}
\]
if $n$ is odd
down the middle column.

The classification in Theorem \ref{t_class} meshes well with 
the general framework of highest weight theory for finite $W$-algebras
developed in \cite{BGK}.
Under this framework for each $A \in \Row_{n,l}$ one can associate
an irreducible $U(\mf{g},e)$-module.  In $\S$\ref{section5} we 
show that this module is isomorphic to $L(A)$ for each $A \in \Row_{n,l}$.

The theorem also helps illuminate the connection
between $U(\mf{g})$-modules and $U(\mf{g},e)$-modules via primitive ideals.
For an algebra $A$ let $\Prim \: A$ denote the set of primitive ideals in $A$.
In \cite{Lo2} Losev showed that there exists a 
surjective map
\[
  \dagger : \Prim_{\operatorname{fin}}U(\mf{g},e) \to \Prim_{\overline{G.e}}U(\mf{g}).
\]
Here 
$G$ is the adjoint group of $\mf{g}$,  
$\Prim_{\operatorname{fin}}U(\mf{g},e)$ denotes the primitive ideals of
$U(\mf{g},e)$ of finite co-dimension, and 
\[
 \Prim_{\overline{G.e}}U(\mf{g}) =
\{ I \in \Prim \: U(\mf{g}) \:|\:  \mathcal{VA}(I) = \overline{G.e} \},
\]
where $\mathcal{VA}(I)$ denotes the associated variety of an
ideal $I$ in $U(\mf{g})$.
Moreover, Losev showed that the
fibers of the map $\dagger$ are $C$-orbits, 
where $C=C_G(e) / C_G(e)^\circ$ is the component
group associated to the nilpotent element $e$, which acts naturally as
automorphisms on $U(\mf{g},e)$ (induced ultimately 
by its adjoint action on $U(\mf{g})$).

In our special cases we can calculate explicitly the action of 
$C$ on the set of 
finite dimensional irreducible $U(\mf{g},e)$-modules, and therefore on
$\Prim_{\operatorname{fin}} U(\mf{g},e)$.
By \cite[Chapter 13]{C} the only rectangular finite $W$-algebras for which $C$ is not trivial
are the ones where $\eps=-$, and $n$ and $l$ are both even, in which
case $C \cong \Z_2$.
To explicitly state the $C$-action 
we need to define the notion of a {\em $\sharp$-special} element of a list of 
complex numbers.
Given a list  $(a_1,\dots,a_{2k+1})$ of complex numbers
let $\{(a_1^{(i)},\dots,a_{2k+1}^{(i)})\:|\:i \in I\}$
be the set of all permutations of this list which satisfy
$a^{(i)}_{2j-1}+a^{(i)}_{2j} >$ 0 for each $j=1,\dots,k$.
Assuming such rearrangements exist, we define the {\em $\sharp$-special  
element} of $(a_1,\dots,a_{2k+1})$
to be the unique maximal element of the set $\{a_{2k+1}^{(i)}\:|\:i  
\in I\}$.
On the other hand, if no such rearrangements exist, we say that the 
$\sharp$-special element of $(a_1,\dots,a_{2k+1})$ is undefined.
For example, the $\sharp$-special element of $(-3, -1, 2)$ is $-3$, whereas  
the $\sharp$-special element of $(-3,-2,1)$ is  
undefined.

We define an action of $\Z_2$ on $\Row_{n,l}$ as follows.
Let $A = (a_{i,j})_{i \in \mathcal{I}_n, j \in \mathcal{I}_l}
\in \Row_{n,l}$, 
let $a$ be the $\sharp$-special element of 
$(0, a_{-1, l-1}, a_{-1, l-3}, \dots, a_{-1, l-1})$, and
let $c$ denote the generator of $\Z_2$.
If $a$ is undefined or $a=0$ then we declare that $c \cdot A = A$.
Otherwise we declare that $c \cdot A = B$ 
where $B \in \Row_{n,l}$ has the same rows as $A$, 
except with one occurrence of $a$ replaced with $-a$ in row $-1$, 
and one occurrence of $-a$ replaced with $a$
in row $1$.
It is an immediate corollary of Lemma \ref{l_sharp0} below that this 
action is well defined.
For example,
\[
  c \cdot 
  \ydiagram{
     {-3\phantom{:}} & {1} & 2 & 4 \cr 
 {-4\phantom{:}}     & -2 & -1 & 3\cr
}
   = 
  \ydiagram{
     {-3\phantom{:}} & {-2} & 1 & 4 \cr 
 {-4\phantom{:}}     & {-1} & 2 & 3 \cr
}.
\]
since the $\sharp$-special element of $(0,-3,1,2,4)$ is $2$.

In $\S$\ref{section6} we prove the following theorem:
\begin{Theorem} \label{t_c}
   Suppose that $n$ and $l$ are even positive integers and
   $\eps=-$.
   Let $A = (a_{i,j})_{i \in \mathcal{I}_n, j \in \mathcal{I}_l}
\in \Row_{n,l}$ be such that 
$A^+ \in \Std_{n,l}^+$ and let
$L(A)$ denote the corresponding finite
dimensional irreducible representation of $U(\mf{g},e)$.  
Then the $\sharp$-special element of 
$(0, a_{-1, l-1}, a_{-1, l-3}, \dots, a_{-1, l-1})$ is defined,
and $c \cdot L(A) = L(c \cdot A)$.
\end{Theorem}

Understanding the $C$-action for the rectangular finite $W$-algebras turns 
out to be key to understanding the $C$-action for more complicated finite $W$-algebras.
In the forthcoming paper \cite{BrG} we use these results as well as the results in \cite{BGK} to classify the finite dimensional irreducible representations of $U(\mf{g},e)$ 
for a large class of nilpotent elements in the symplectic and orthogonal Lie algebras.

\vspace{.1 in}

\noindent {\em Acknowledgements.} The author would like to thank Jonathan Brundan for many enlightening conversations,
and he would like to thank Simon Goodwin for pointing out
and correcting 
an error in $\S$\ref{section6}.

\section{Rectangular finite $W$-algebras}

\subsection{Overview of finite $W$-algebras} \label{section2}
Throughout this subsection $\mf{g}$ denotes a reductive Lie algebra and
$e$ denotes a nilpotent element of $\mf{g}$.
To define the finite $W$-algebra $U(\mf{g},e)$, one first applies
the Jacobson-Morozov Theorem to embed $e$
into an $\mf{sl}_2$-triple $(e,h,f)$.  Now the $\ad\ h$ eigenspace 
decomposition gives a grading on 
$\mf{g}$: 
\begin{equation} \label{eq_grading}
  \mf{g} = \bigoplus_{i \in \Z} \mf{g}(i),
\end{equation}
where $\mf{g}(i) = \set{ x \in \mf{g} }{ [h,x] = i x}$.
Finite $W$-algebras are defined for any grading, however
to simplify the definition of $U(\mf{g},e)$, we
assume that this grading is an even grading, i.e., $\mf{g}(i) = 0$ if
$i$ is odd.
Define a character $\chi : \mf{g} \to \C$ by
$\chi(x) = (x,e)$, where $(.\,,.)$ is a fixed non-degenerate symmetric
invariant bilinear form on $\mf{g}$.
Let $\mf{m} = \bigoplus_{i < 0} \mf{g}(i)$,
and let 
$\mf{p} = \bigoplus_{i \ge 0} \mf{g}(i)$.
Let $I$ be the left ideal of $U(\mf{g})$ generated by
$\set{m - \chi(m) }{ m \in \mf{m}}$.
By the PBW Theorem, 
\begin{equation} \label {eq_decomp}
  U(\mf{g}) = U(\mf{p}) \oplus I.
\end{equation}
Define $\pr : 
U(\mf{g}) \to U(\mf{p})$ to be the 
projection along this direct sum decomposition.
Now we define
\[
U(\mf{g},e)
= \set{u \in U(\mf{p}) }{ \pr([m,u]) = 0\text{ for all }m \in \mathfrak{m} },
\] 
so  $U(\mf{g},e)$ is a subalgebra of $U(\mf{p})$ in these even grading cases.

The finite $W$-algebra $U(\mf{g},e)$
possesses two natural filtrations.
The first of these, the {\em Kazhdan filtration},
is the filtration on $U(\mf{g},e)$ induced by the filtration
on $U(\mathfrak{g})$ generated by declaring that
each element $x \in \mf{g}(i)$ in the grading (\ref{eq_grading})
is of
degree $i+2$.
The fundamental {\em PBW theorem} for finite $W$-algebras
asserts that the associated
graded algebra to $U(\mf{g},e)$
under the Kazhdan filtration
is canonically isomorphic to the coordinate
algebra of the Slodowy slice  at
$e$; see e.g. \cite[Theorem 4.1]{GG}.

The second important filtration is called the {\em good filtration}.
The good filtration is the filtration induced
on $U(\mf{g},e)$ by the grading \eqref{eq_grading}
on $U(\mf{p})$.
According to this definition, the associated graded algebra $\gr U(\mf{g},e)$
is identified with a graded subalgebra of $U(\mf{p})$.
The fundamental result about the good filtration, which is a
consequence of the PBW theorem and \cite[(2.1.2)]{P2}, is that
\begin{equation} \label{assgr}
\gr U(\mf{g},e) = U(\mf{g}^e)
\end{equation}
as graded subalgebras of $U(\mathfrak{p})$,
where $\mf{g}^e$ denotes the centralizer of $e$ in $\mf{g}$;
see also \cite[Theorem 3.5]{BGK}.

\subsection{Rectangular finite $W$-algebras and twisted Yangians} \label{section2.1}
Recall that a rectangular finite $W$-algebra is a finite $W$-algebra 
$U(\mf{g},e)$ for which $\mf{g}$ is $\mf{so}_n(\C)$ 
or $\mf{sp}_n(\C)$ and $e$ has Jordan blocks all the same size.
We need to recall the many of the results from \cite{Br} 
about the relationship between twisted Yangians and
rectangular finite $W$-algebras.
We begin by fixing explicit matrix realizations for the
classical Lie algebras. 
Recall that for any integer $n \geq 1$, we 
have defined the index set $\mathcal{I}_n = \{1-n, 3-n, \dots, n-1\}$.
Let $\mathfrak{g}_n = \mathfrak{gl}_n(\C)$ 
with standard basis given by the matrix units $\{e_{i,j} \:|\: i,j \in \mathcal{I}_n\}$.
Let $J_n^+$
be the $n \times n$ matrix with $(i,j)$ entry equal to $\delta_{i,-j}$, and
set
$$
\mf{g}_n^{+} = 
\mf{so}_n(\C)=
\{x\in \mf{g}_n\:|\:
x^T J_n^{+} + J_n^{+} x = 0\},
$$
where $x^T$ denotes the usual transpose of an $n\times n$ matrix.
Assuming in addition that $n$ is even, let
$J_n^-$ be the $n \times n$ matrix with $(i,j)$ entry equal to
$\delta_{i,-j}$ if $j > 0$ and $-\delta_{i,-j}$ if $j < 0$, and set
\begin{equation} \label{eq_J}
\mf{g}_n^{-} = 
\mf{sp}_n(\C)
=\{x \in \mf{g}_n\:|\:
x^T J_n^{-} + J_n^{-} x = 0\}.
\end{equation}

We 
adopt the following convention regarding signs.
For $i \in \mathcal I_n$, define
$\hat \imath \in \Z /2 \Z$ by
\begin{equation}
\hat \imath = 
\begin{cases}
0&\text{if $i \geq 0$;}\\
1&\text{if $i < 0$.}
\end{cases}
\end{equation}
We will often identify a sign $\pm$ 
with the integer
$\pm 1$ when writing formulae.
For example, $\eps^{\hat \imath}$ denotes $1$ if $\eps = +$ or $\hat \imath = 0$,
and it denotes $-1$ if $\eps = -$ and $\hat \imath = 1$.
With this notation, $\mf{g}_n^\eps$ is spanned by the matrices
$\{ f_{i,j} \:|\:i,j \in \mathcal{I}_n\}$, where
\[
  f_{i,j} = e_{i,j} - \eps^{\hat \imath + \hat \jmath} e_{-j,-i}.
\]

Next
we fix integers
$n, l \geq 1$ and signs $\epsilon, \phi \in \{\pm\}$, assuming
that 
$\phi = \epsilon$ if $l$ is odd,
$\phi =- \epsilon$ if $l$ is even, and
$\phi = +$ if $n$ is odd;
now let $\mathfrak{g} = \mathfrak{g}_{nl}^\epsilon$.
To define a nilpotent element $e \in \mf{g}$ of Jordan type
$(l^n)$ we
introduce an $n \times l$ rectangular array of boxes, labeling
rows in order from top to bottom by the index set $\mathcal I_n$ and 
columns in order from left to right by the index set $\mathcal I_l$.
Also label the individual boxes in the array with the elements of the set
$\mathcal I_{nl}$. For $a \in \mathcal I_{nl}$ we let
$\row(a)$ and $\col(a)$ denote the row and column numbers of the box in which 
$a$ appears. We require that the boxes are labeled skew-symmetrically
in the sense that $\row(-a) = -\row(a)$ and $\col(-a) = -\col(a)$.
If $\epsilon = -$
we require in addition that $a > 0$ either if
$\col(a) > 0$ or if $\col(a) = 0$ and $\row(a) > 0$;
this additional restriction streamlines
some of the signs appearing in formulae below. 
For example, if $n= 3, l = 2$ and $\epsilon = -, \phi = +$, 
one could pick the labeling
\[
  \ydiagram{
     {-5\phantom{:}} & {1} \cr 
 {-3\phantom{:}}     & {3} \cr
 {-1\phantom{:}}     & {5} \cr
   }
\]
and get that $\row(1) = -2$ and $\col(1) = 1$.
We remark that the above arrays are a special case of the {\em pyramids}
introduced by Elashvili and Kac in \cite{EK}; see also \cite{BG}.

Having made these choices, we let $e \in \mathfrak{g}$ denote the
following nilpotent matrix of Jordan type $(l^n)$:
\begin{equation} \label{eq_e}
  e = 
     \sum_{
        \substack{a,b \in \mathcal{I}_{nl} \\ \row(a) = \row(b) \\ 
           \col(a) +2= \col(b) \geq 2} }
       f_{a,b}
+
\sum_{
     \substack{a,b \in \mathcal{I}_{nl} \\ 
       \row(a) = \row(b) > 0\\ \col(a)+2=\col(b)=1}}
       f_{a,b}
+
\sum_{
     \substack{a,b \in \mathcal{I}_{nl} \\ 
            \row(a) = \row(b) = 0 \\ 
            \col(a)+2=\col(b)=1 } }
       \textstyle{\frac{1}{2}}f_{a,b}.
\end{equation}
In the above example,
$e=f_{-1,5} + \frac{1}{2} f_{-3,3} = e_{-1,5} + e_{-5,1} + e_{-3,3}$.
Also define an even grading 
\begin{equation} \label{grading}
\mf{g} = \bigoplus_{r \in \Z} \mf{g} (r)
\end{equation}
with $e \in \mf{g}(2)$
by declaring that $\deg(f_{a,b}) = \col(b) - \col(a)$.
Note this grading coincides with the grading obtained
by embedding $e$ into the $\mf{sl}_2$-triple $(e,h,f)$
where
\begin{equation} \label{eq_h}
  h = \sum_{a \in \mathcal{I}_{nl}} \col(-a) e_{a,a}
\end{equation}
and considering the
$\operatorname{ad} h$-eigenspace decomposition of $\mf{g}$.  
Let $\mf{p} = \bigoplus_{r \ge 0} \mf{g}(r)$
and $\mf{m} = \bigoplus_{r < 0} \mf{g}(r)$.
For the non-degenerate symmetric invariant bilinear form on $\mf{g}$
we use the form $(x,y) = \frac{1}{2} \trace(x y)$.
Define $\chi : \mf{m} \to \C$ by $x \mapsto (e, x)$.
An explicit calculation using the formula for the nilpotent matrix $e$ 
recorded above
shows that
\begin{equation} \label{prchi}
\chi(f_{a,b}) = -\eps^{\hat a + \hat b} \chi(f_{-b,-a}) = 1
\end{equation}
if $\row(a) = \row(b),
\col(a) = \col(b)+2$ and either $\col(a) \geq 2$ or
$\col(a) = 1$, $\row(a) \geq 0$; all other $f_{a,b} \in \mf{m}$ satisfy 
$\chi(f_{a,b}) = 0$.
Now we have our rectangular finite $W$-algebra
\[
U(\mf{g},e)
= \set{u \in U(\mf{p}) }{ \pr([x,u]) = 0\text{ for all }x \in \mathfrak{m} }
\] 
where $\pr : U(\mf{g}) \to U(\mf{p})$ is projection along the the decomposition from
\eqref{eq_decomp}.

To make the connection between $U(\mf{g},e)$ and the twisted Yangians, we 
exploit a shifted version of the Miura
transform, which we define as follows.
Let $\mf{h} = \mf{g}(0)$ be the Levi factor 
of $\mf{p}$ coming from
the grading. 
It is helpful
to bear in mind that
there is an isomorphism
\begin{equation}\label{h}
\qquad\:\:\mathfrak{h} \cong \begin{cases}
\mathfrak{g}_n^{\oplus m}&\text{if $l = 2m$;}\\
\mathfrak{g}_n^\eps \oplus \mathfrak{g}_n^{\oplus m}&\text{if $l = 2m+1$,}
\end{cases}
\end{equation}
which maps $f_{a,b} \in \mathfrak{h}$
to $f_{\row(a), \row(b)}\in \mathfrak{g}_n^\eps$ if $\col(a)=\col(b)=0$ or to
$e_{\row(a),\row(b)}$ in the $\lceil \frac{\col(a)}{2}\rceil$th copy of
$\mathfrak{g}_n$ if $\col(a)=\col(b)>0$.
For $q \in \mathcal{I}_l$, let
\begin{equation} \label{rho}
  \rho_q = 
   \begin{cases}  
      \frac{n q  - \epsilon}{2} & \text{if $ q > 0$;} \\
      \frac{n q + \epsilon}{2} & \text{if $ q < 0$;} \\
      0 & \text{if $q=0$}.
   \end{cases}
\end{equation}
Let $\eta$ be the automorphism of $U(\mf{h})$ defined on generators by
$\eta(f_{a,b}) = f_{a,b} - \delta_{a,b} \rho_{\col(a)}$.
Let $\xi : U(\mf{p}) \twoheadrightarrow U(\mf{h})$ be the algebra homomorphism induced by the
natural projection $\mf{p} \twoheadrightarrow \mf{h}$.  
The {\em Miura transform} $\mu: U(\mf{p}) \rightarrow U(\mf{h})$
is the composite map
\begin{equation} \label{miura}
  \mu = \eta \circ \xi. 
\end{equation}
By \cite[$\S$2.3]{Ly} (or \cite[Theorem 3.4]{Br}) the restriction of $\mu$ 
to $U(\mf{g},e)$ is injective.

Now we turn our attention to the twisted Yangian $Y=Y_n^\phi$,
recalling that $\phi = -\eps$ if $l$ is even and $\phi = \eps$ if $l$ is odd.
By definition, $Y$ is a subalgebra of the Yangian $Y_n$.
The Yangian $Y_n$
is a Hopf algebra over $\C$ with countably many generators 
$\set{T_{i,j}^{(r)} }{ i,j \in \mathcal{I}_n, r \in \Z_{>0}}$.
To give the defining relations and other data for the Yangian
it is convenient to use the power series
\[
  T_{i,j}(u) = \sum_{r \ge 0} T_{i,j}^{(r)}u^{-r} \in Y_n[[u^{-1}]]
\]
where $T_{i,j}^{(0)} = \delta_{i,j}$. 
Now the defining relations are
\[
  (u-v) [T_{i,j}(u), T_{k,l}(u)] 
     = T_{k,j}(u) T_{i,l}(v) - T_{k,j}(v) T_{i,l}(u).
\]
This and subsequent formulae involving generating functions
should be interpreted by
equating coefficients of the indeterminates $u$ and $v$ 
on both sides of equations,
as discussed in detail in \cite[$\S$1]{MNO}.  
For example,
the comultiplication 
$\Delta:Y_n \rightarrow Y_n \otimes Y_n$ making $Y_n$ into a Hopf algebra 
is defined by the
formula
\begin{equation}\label{com0}
  \Delta(T_{i,j}(u)) = \sum_{k \in \mathcal{I}_n}
     T_{i,k}(u) \otimes T_{k,j}(u).
\end{equation}

By  \cite[$\S$3.4]{MNO}, there exists an 
automorphism $\tau:Y_n \rightarrow Y_n$
of order $2$
defined by 
$$
\tau(T_{i,j}(u)) = \phi^{\hat \imath + \hat \jmath} T_{-j,-i}(-u).
$$
We define the twisted Yangian $Y$ to be the subalgebra of $Y_n$ 
generated by the elements $\{S_{i,j}^{(r)} \:|\: i,j \in \mathcal{I}_n, r \in \Z_{>0}\}$
coming from the expansion
\begin{equation}\label{siju}
  S_{i,j} (u) = \sum_{r \geq 0} S_{i,j}^{(r)} u^{-r}
= 
\sum_{k \in \mathcal{I}_n} 
    \tau(T_{i,k}(u)) T_{k,j}(u) \in Y_n[[u^{-1}]].
\end{equation}
This is not the same embedding of $Y$ into $Y_n$ as used in
\cite[$\S$3]{MNO}: we have twisted the embedding there 
by the automorphism $\tau$.
The relations for the twisted Yangian are given by
\begin{align} \label{eq_y_rel1}
 (u^2-v^2)  [S_{i,j}(u),S_{k,l}(v)] &= 
     (u+v) (S_{k,j}(u) S_{i,l}(v) - S_{k,j}(v) S_{i,l}(u)) - \\
     & \quad (u-v) (\phi^{\hat k + \widehat{- \jmath}} S_{i,-k}(u) S_{-j,l}(v) -
            \phi^{\hat k + \widehat{-l} } S_{k,-i}(v) S_{-l,j}(u)) + \notag\\
     & \quad 
           \phi^{\hat i + \widehat{-\jmath}} S_{k,-i}(u) S_{-j,l}(v) -
            \phi^{\hat i + \widehat{-\jmath} } S_{k,-i}(v) S_{-j,l}(u) \notag
\end{align}
and
\begin{equation} \label{eq_y_rel2}
  \phi^{\hat i + \hat j} S_{-j,-i}(-u) = 
    S_{i,j}(u) + \phi \frac{S_{i,j}(u) - S_{i,j}(-u)}{2u}.
\end{equation}

Because of the fact that $\tau$ is a coalgebra antiautomorphism
of $Y_n$, we get from \cite[$\S$4.17]{MNO} that the restriction
of $\Delta$ to $Y$ has image contained in $Y \otimes Y_n$
and
\begin{equation}\label{com}
\Delta(S_{i,j}(u)) = \sum_{h,k \in \mathcal{I}_n} S_{h,k}(u) 
\otimes \tau(T_{i,h}(u)) T_{k,j}(u).
\end{equation}
We let 
$\Delta^{(m)}:Y_n \rightarrow Y_n^{\otimes (m+1)}$ 
denote the $m$th iterated comultiplication.
The preceding formula shows that it maps
$Y$ into $Y \otimes Y_n^{\otimes m}$.

By \cite[$\S$1.16]{MNO} there is an {evaluation homomorphism}
$Y_n \rightarrow U(\mathfrak{g}_n)$.
In view of this and (\ref{h}), we 
obtain for every $0 < p \in \mathcal{I}_l$ 
a homomorphism
\begin{equation}\label{evk}
\ev_p:Y_n \rightarrow U(\mathfrak{h}),
\qquad
T_{i,j}(u) \mapsto \delta_{i,j} +u^{-1} f_{a,b},
\end{equation}
where $a,b \in \mathcal{I}_{nl}$ are defined
from $\row(a) = i, \row(b) = j$ and $\col(a)=\col(b) = p$.
The image of this map 
is contained in the subalgebra of $U(\mathfrak{h})$
generated by the $\lceil p/2 \rceil$th copy of $\mathfrak{g}_n$ from the
decomposition (\ref{h}).
There is also an evaluation homomorphism
$Y \rightarrow U(\mathfrak{g}_n^\phi)$
defined in \cite[$\S$3.11]{MNO}.
If we assume that $l$ is odd (so $\eps = \phi$),
we can therefore define another homomorphism
\begin{equation}\label{ev0}
\ev_0:Y \rightarrow U(\mathfrak{h}),
\qquad
S_{i,j}(u) \mapsto \delta_{i,j} + (u+{\textstyle\frac{\phi}{2}})^{-1} f_{a,b},
\end{equation}
where 
$\row(a) = i, \row(b) = j$ and $\col(a) = \col(b) = 0$;
if $\eps = -$ this depends on our convention for labeling boxes as specified 
above.
The image of this map is contained in the 
subalgebra of $U(\mathfrak{h})$
generated by the subalgebra
$\mathfrak{g}_n^\eps$ in the decomposition (\ref{h}).
Putting all these things together, we deduce that there is 
a homomorphism
$$
\kappa_l:Y \rightarrow U(\mathfrak{h})
$$
defined by
\begin{equation} \label{kappa}
\kappa_l = 
\begin{cases}
\ev_1 \bar\otimes \ev_3 \bar\otimes\cdots\bar\otimes \ev_{l-1}
\circ \Delta^{(m)}&\text{if $l=2m+2$;}\\
\ev_0 \bar\otimes \ev_2 \bar\otimes\cdots\bar\otimes \ev_{l-1}
\circ \Delta^{(m)}&\text{if $l=2m+1$,}
\end{cases}
\end{equation}
where $\bar\otimes$ indicates composition with the natural multiplication 
in $U(\mathfrak{h})$.

Theorem \ref{th_ch3} is a corollary of the following theorem:
\begin{Theorem}[{\cite[Theorem 1.1]{Br}}]
\label{main1} 
$\mu(U(\mf{g},e))
=
\kappa_l(Y).$
\end{Theorem}

This implies the following:
\begin{Corollary} \label{pluswalg}
When $l$ is even there is a surjection
\[
  \zeta : U(\mf{g'}, e') \twoheadrightarrow 
   U(\mf{g}, e)
\]
where $\mf{g}' = \mf{g}_{n(l+1)}^{-\eps}$ and $e'$ is a nilpotent element 
in  $\mf{g}'$
of Jordan type $((l+1)^n)$ such that the following diagram commutes:
\[
\begin{diagram}
Y  & \rTo^{\kappa_{l+1}}   & U(\mf{g}', e')  \\
             & \rdTo_{\kappa_l} & \dTo_{\zeta} \\
&  &  U(\mf{g}, e)
\end{diagram}
\]
\end{Corollary}
Note that this corollary does not apply when $\eps = +$ and $n$ is odd since
in this case the nilpotents $e$ and $e'$ do not exist.

The proof of Theorem~\ref{main1}
requires
an explicit formula for the generators of 
$U(\mf{g},e)$ corresponding to the elements $S_{i,j}^{(r)}\in Y$,
which we will use again later on.
Given $i,j \in \mathcal{I}_n$ 
and $p,q \in \mathcal{I}_l$,
let $a,b$ be the elements of $\mathcal{I}_{n l}$ such that
$\col(a) = p$,
$\col(b) = q$, 
$\row(a) = i$, and
$\row(b) = j$.
Define a linear map
$s_{i,j} : \mf{g}_l \rightarrow \mf{g}$ by setting
\begin{equation} \label{sij}
    s_{i,j}(e_{p,q}) = \phi^{\hat \imath \hat p + \hat \jmath \hat q} 
f_{a,b}.
\end{equation}
Let $M_n$ denote the algebra of $n \times n$ matrices
over $\C$, with rows
and columns labeled by the index set $\mathcal{I}_n$ as usual,
and let $T(\mf{g}_l)$ be the tensor algebra on the vector space
$\mf{g}_l$.
Let 
\begin{equation} \label{S}
  s : T(\mf{g}_l) \to M_n \otimes U(\mf{g})
\end{equation}
be the algebra homomorphism that maps a generator
$x \in \mf{g}_l$ to
$\sum_{i,j \in \mathcal{I}_n} e_{i,j} \otimes s_{i,j}(x)$.
This in turn defines linear maps
\begin{equation} \label{eq_sij}
  s_{i,j} :  T(\mf{g}_l) \to U(\mf{g})
\end{equation}
such that
\[
s(x) = \sum_{i,j \in \mathcal{I}_n} e_{i,j} \otimes s_{i,j}(x)
\]
for every $x \in T(\mf{g}_l)$.
Note for any $x, y \in T(\mf{g}_l)$ that
\begin{equation} \label{sijmult}
  s_{i,j}(x y) = \sum_{k \in \mathcal{I}_n} s_{i,k}(x) s_{k,j}(y)
\end{equation}
and also
$s_{i,j}(1) = \delta_{i,j}$.

If $A$ is an $l \times l$ 
matrix with entries in some ring,
we define its {\em row determinant} $\rdet A$
to be the usual Laplace expansion of determinant, but keeping
the (not necessarily commuting) monomials that arise in {\em row order}; 
see e.g. \cite[(12.5)]{BKshifted}.
For $q \in \mathcal{I}_l$ and an indeterminate $u$,
let 
$$
u_q = u+ e_{q,q} +\rho_q \in T(\mf{g}_l)[u],
$$
recalling the definition of $\rho_q$ from \eqref{rho}.
Define $\Omega(u)$ to be the $l \times l$ matrix with
entries in $T(\mf{gl}_l)[u]$ whose $(p,q)$ entry for $p,q \in \mathcal I_l$
is equal to
\begin{equation}\label{omega_even}
\Omega(u)_{p,q} =
\left\{
\begin{array}{ll}
e_{p,q}&\text{if $p < q$;}\\
u_q&\text{if $p=q$;}\\
-1&\text{if $p = q+2 < 0$;}\\
-\phi&\text{if $p  = q+2 = 0$;}\\
1&\text{if $p = q+2 > 0$;}\\
0&\text{if $p  > q+2$.}
\end{array}\right.
\end{equation}
If $l$ is odd we also need the $l \times l$ matrix
$\bar\Omega(u)$ defined by
\begin{equation} \label{bar_omega}
  \bar\Omega(u)_{p,q} = 
   \begin{cases}
     \Omega(u)_{p,q} & \text{if $p \ne 0$ or $q \ne 0$;} \\
     e_{0,0}   & \text{if $p = q = 0$.}
    \end{cases}
\end{equation}
See \cite[$\S$1]{Br} for examples of $\Omega(u)$ and $\bar \Omega(u)$.
Now let
\begin{equation}\label{omegadef}
  \omega(u) = 
\sum_{r = -\infty}^{l} \omega_{l-r} u^{r}  = 
\left\{
\begin{array}{ll}
\rdet \Omega(u)&\text{if $l$ is even;}\\
\displaystyle\rdet \Omega(u) + \sum_{r=1}^{\infty} (-2\phi u)^{-r} 
\rdet \bar\Omega(u)  &\text{if $l$ is odd.}
\end{array}\right.
\end{equation}
This defines elements $\omega_r \in T(\mathfrak{g}_l)$,
hence elements $s_{i,j}(\omega_r) \in U(\mathfrak{g})$
for $i,j \in \mathcal{I}_n$ and $r \geq 1$.
It is obvious from the definition that
each $s_{i,j}(\omega_r)$ actually belongs to $U(\mathfrak{p})$.

\begin{Theorem}
[{\cite[Theorem 1.2]{Br}}]
\label{gens}
The elements 
$\{s_{i,j}(\omega_r)\:|\:i, j \in \mathcal{I}_n, r \geq 1\}$
generate the subalgebra $U(\mf{g},e)$. Moreover,
$\mu(s_{i,j}(\omega_r)) = \kappa_l(S_{i,j}^{(r)})$.
\end{Theorem}
It will be useful to note this theorem implies that
for all $i,j \in \mathcal{I}_n$
\begin{equation} \label{seriesgens}
s_{i,j}(u^{-l} \omega(u)) = \kappa_l(S_{i,j}(u)).
\end{equation} 

\section{Representation theory of Yangians and twisted Yangians}

To prove Theorem \ref{t_class} we need to review 
the representation theory of Yangians and twisted Yangians from \cite{M}.

\subsection{Representation theory of Yangians}

We say a $Y_n$-module $V$ is a {\em highest weight module} if
it is generated by a vector $v$ such that 
$T_{i,j}(u)v = 0$ for all $i < j$, 
and if for all $i$ we have that $T_{i,i}(u)v = \lambda_i(u)v$ 
for some power series $\lambda_i \in 1 + u^{-1} \C[[u^{-1}]]$,
in which case we say that $V$ is of highest weight
\begin{equation} \label{barlambda}
\bar \lambda(u) = (\lambda_{1-n}(u), \lambda_{3-n}(u), \dots, \lambda_{n-1}(u)).
\end{equation}
For the rest of this paper
we consider $\bar \lambda(u) \in (1 + u^{-1}\C[[u^{-1}]])^n$ to be indexed
by the set $\mathcal{I}_n$ as in \eqref{barlambda}.

The following theorem is contained in \cite[$\S$2]{M}.
\begin{Theorem}
For each weight
$
  \bar \lambda(u) \in
  (1 + u^{-1}\C[[u^{-1}]])^n
$
there is a unique (up to isomorphism) irreducible 
highest weight $Y_n$-module 
$L(\bar \lambda(u))$
of highest weight $\bar \lambda(u)$.
\end{Theorem}

\begin{Theorem}[{ \cite[Theorem 2.3]{M}}]
Every irreducible finite dimensional $Y_n$-module is a highest weight module.
\end{Theorem}

To specify which irreducible highest weight modules
are finite dimensional, following Molev, we introduce the following notation.
Given two power series $\lambda_1(u), \lambda_2(u) \in 1+u^{-1} \C[[u^{-1}]]$ we
write $\lambda_1(u) \rightarrow \lambda_2(u)$  
if there exists a monic polynomial
$P(u) \in \C[u]$ such that 
\[
  \frac{\lambda_1(u)}{\lambda_2(u)} = \frac{P(u+1)}{P(u)}.
\]
In fact $P(u)$ must then be unique because if $Q(u)$ is another monic polynomial
satisfying
$
  \frac{\lambda_1(u)}{\lambda_2(u)} = \frac{Q(u+1)}{Q(u)}
$
then
$
\frac{Q(u)}{P(u)} = \frac{Q(u+1)}{P(u+1)},
$
thus $\frac{Q(u)}{P(u)}$ is periodic, which implies $P(u)=Q(u)$.

Here is the main classification theorem for finite dimensional
irreducible representations of $Y_n$.
\begin{Theorem}
[Drinfeld, {\cite{Dr}}]
  \label{t_drin_y_n_class}
The $Y_n$-module 
$L(\bar \lambda(u))$ is finite
dimensional if and only if
$\lambda_{1-n}(u) \rightarrow \lambda_{3-n}(u) \rightarrow \dots \rightarrow \lambda_{n-1}(u)$.
\end{Theorem} 

The following lemmas give a more combinatorial description of this notation.
Recall that $\ge$ denotes the partial order on $\C$ where 
$a \ge b$ if $a-b \in \Z_{\ge 0}$.

\begin{Lemma} \label{larrow}
If $\lambda_1(u), \lambda_2(u) \in 1 + u^{-1} \C[[u^{-1}]]$ then
$\lambda_1(u) \rightarrow \lambda_2(u)$ if and only if there exists  
$\gamma(u) \in 1+ u^{-1}\C[[u^{-1}]]$ such that 
\begin{align*}
  \gamma(u) \lambda_1(u) &= (1+a_1 u^{-1}) \dots (1+a_k u^{-1}), \\
  \gamma(u) \lambda_2(u) &= (1+b_1 u^{-1}) \dots (1+b_k u^{-1})
\end{align*}
where
$a_i \ge b_i$ for $i = 1, \dots, k$.
\end{Lemma}
\begin{proof}
First assume that $\lambda_1(u) \rightarrow \lambda_2(u)$, so there exists
a monic polynomial $P(u)$ such that
\[ 
  \frac{\lambda_1(u)}{\lambda_2(u)} = \frac{P(u+1)}{P(u)}.
\]
Let $k$ be the degree of $P(u)$,
and let $\gamma(u) = \frac{P(u) u^{-k}}{\lambda_2(u)}$ .
So
$\gamma(u) \lambda_1(u) = P(u+1)u^{-k}$ and
$\gamma(u) \lambda_2(u) = P(u) u^{-k}$,
thus $\gamma(u)$
satisfies the
conclusions
of the lemma since 
we can now write 
$\gamma(u) \lambda_2(u) u^k = P(u) =
(u+b_1 ) \dots (u+b_k )
$
and
$\gamma(u) \lambda_1(u) u^k= P(u+1)
=
(u+b_1 +1) \dots (u+b_k +1)$.

Now assume there exists $\gamma(u) \in 1 + u^{-1}\C[[u^{-1}]]$ such that
$\gamma(u) \lambda_1(u) = (1+a_1 u^{-1}) \dots (1+a_k u^{-1})$ and
$\gamma(u) \lambda_2(u) = (1+b_1 u^{-1})\dots (1+b_k u^{-1})$
where $a_i \ge b_i$ for  $i=1, \dots, k$.
For  $i=1, \dots, k$ let
$P_i(u) = (u+a_i-1)(u+a_i-2) \dots (u+b_i+1)$.
Now 
\begin{align*}
\frac{\lambda_1(u)}{ \lambda_2(u)} = 
\frac{(u+a_1)P_1(u) \dots (u+a_k)P_k(u)}
{P_1(u)(u+b_1) \dots P_k(u) (u+b_k)},
\end{align*}
so $P(u) = P_1(u)(u+b_1) \dots P_k(u) (u+b_k)$
is the unique polynomial 
satisfying 
\[
 \frac{\lambda_1(u)}{\lambda_2(u)} = \frac{P(u+1)}{P(u)}.
\]
\end{proof}

\begin{Lemma} \label{lpoly}
Let $\lambda_1(u), \lambda_2(u) \in 1+u^{-1} \C[[u^{-1}]]$, and suppose that
$\lambda_1(u) \rightarrow \lambda_2(u)$.
If $\gamma(u) \in 1+u^{-1}\C[[u^{-1}]]$ satisfies
$\gamma(u) \lambda_1(u), \gamma(u) \lambda_2(u) \in \C[u^{-1}]$ then we
can write
$\gamma(u) \lambda_1(u) = (1+a_1 u^{-1}) \dots (1+a_k u^{-1})$ and
$\gamma(u) \lambda_2(u) = (1+b_1 u^{-1}) \dots (1+b_k u^{-1})$ where
$a_i \ge b_i$ for  $i=1, \dots, k$.
\end{Lemma}
\begin{proof}
We can write
$\gamma(u) \lambda_1(u) = (1+a_1 u^{-1}) \dots (1+a_k u^{-1})$ and
$\gamma(u) \lambda_2(u) = (1+b_1 u^{-1}) \dots (1+b_k u^{-1})$,
and by replacing $\gamma(u)$ we may assume that the sets
 $\{a_1, \dots, a_k\}$ and $\{b_1, \dots, b_k\}$ 
are disjoint.
By Lemma \ref{larrow} there exists $\gamma'(u) \in 1+ u^{-1} \C[[u^{-1}]]$
such that
$\gamma'(u) \lambda_1(u) = (1+c_1 u^{-1}) \dots (1+c_m u^{-1})$ and
$\gamma'(u) \lambda_2(u) = (1+d_1 u^{-1}) \dots (1+d_m u^{-1})$
where $c_i \ge d_i$ for $i=1, \dots, m$,
and by replacing $\gamma'(u)$ we may assume that the sets
 $\{c_1, \dots, c_m\}$ and $\{d_1, \dots, d_m\}$ 
are disjoint.
So we have that
\[
\frac{(1+a_1 u^{-1}) \dots (1+a_ku^{-1})} 
{(1+b_1 u^{-1} ) \dots (1+b_k u^{-1})}
=
\frac{(1+c_1 u^{-1}) \dots (1+c_m u^{-1})} 
{(1+d_1u^{-1}) \dots (1+d_mu^{-1})}.
\]
So $k=m$, and as unordered sets we have
$(a_1, \dots, a_k) = (c_1, \dots, c_k)$,
and
$(b_1, \dots, b_k) = (d_1, \dots, d_k)$.
Thus the lemma follows by re-indexing
$(a_1, \dots, a_k) $ and $(b_1, \dots, b_k) $.
\end{proof}

\begin{Lemma} \label{lcompat}
Let $\lambda_1(u), \dots \lambda_m(u) \in 1 + u^{-1} \C[[u^{-1}]]$.
If $\lambda_1(u) \rightarrow \lambda_2(u) \rightarrow \dots \rightarrow
\lambda_m(u)$
then there exists  $\gamma(u) \in 1 + u^{-1}\C[[u^{-1}]]$
such that $\gamma(u) \lambda_i(u) \in \C[u^{-1}]$ for $i = 1, \dots, m$.
\end{Lemma}
\begin{proof}
Assume that 
$\lambda_1(u) \rightarrow \lambda_2(u) \rightarrow \dots \rightarrow
\lambda_m(u)$, and for $i = 1, \dots, m-1$ let $P_i(u)$ be the monic
polynomial
so that 
\[
\frac{\lambda_i(u)}{\lambda_{i+1}(u)} = \frac{P_i(u+1)}{P_i(u)}.
\]
Note  for $i=1, \dots, m-1$ that
\[
  \lambda_i(u) = \frac{P_i(u+1) P_{i+1}(u+1) \dots P_{m-1}(u+1)
    \lambda_m(u)}
    {P_i(u) P_{i+1}(u) \dots P_{m-1}(u)}.
\]
So
\[
\gamma(u) = \frac{u^{-k} P_1(u) \dots P_{m-1}(u)}{\lambda_m(u)},
\] where $k= \sum_{i=1}^{m-1} \deg(P_i(u))$,
satisfies the conclusion of the lemma.
\end{proof}

\subsection{Representation theory of twisted Yangians}
Recall that 
a $Y$-module $V$ is called a {\em highest weight module} if
it generated by a vector $v$ such that $S_{i,j}(u) v = 0$ for all $i
< j$ and 
if for all $i$ we have that 
$S_{i,i}(u) v = \mu_i(u) v$ 
for some power series 
$\mu_i(u) \in 1 + u^{-1}\C[[u^{-1}]]$. 
The following theorem is contained in \cite[Chapter 3]{M}.
\begin{Theorem} \label{th_tyhw1}
For a weight 
$\bar \mu(u) = (\mu_{1}(u), \mu_{3}(u), \dots, \mu_{n-1}(u)) \in (1+u^{-1}\C[[u^{-1}]])^{n/2}$ 
if $n$ is even or
$\bar \mu(u) = (\mu_{0}(u), \mu_{2}(u), \dots, \mu_{n-1}(u))  \in (1+u^{-1}\C[[u^{-1}]])^{(n+1)/2}$ if
$n$ is odd,
there is a unique (up to isomorphism) irreducible 
highest weight $Y$-module $L(\bar \mu(u))$ of highest weight $\bar \mu(u)$.
\end{Theorem}

For the rest of this paper we consider 
$\bar \mu(u) \in (1+u^{-1}\C[[u^{-1}]])^{n/2}$ if $n$ is even or
$\bar \mu(u) \in (1+u^{-1}\C[[u^{-1}]])^{(n+1)/2}$ if $n$ is odd
to be indexed by the sets $\{1, 3, \dots, n-1\}$ and 
$\{0,2, \dots, n-1\}$, respectively, as in Theorem \ref{th_tyhw1}.

The following is part of \cite[Theorem 3.3]{M}.
\begin{Theorem}  \label{th_tyhw2}
Every irreducible finite dimensional $Y$-module 
is a highest weight module.
\end{Theorem}

Following Molev, to specify which irreducible highest weight modules
are finite dimensional, we introduce the following notation.
For power series $\mu(u), \nu(u) \in 1 + u^{-1}\C[[u^{-1}]]$, we write
$\mu(u) \Rightarrow \nu(u)$ if there exists a monic polynomial
$P(u) \in \C[u]$ such that $P(u) = P(1-u)$ and
\[
  \frac{\mu(u)}{\nu(u)} = \frac{P(u+1)}{P(u)}.
\]
Note that $P(u) = P(1-u)$ is equivalent to $P(u)$ being of even degree 
and the roots of $P(u)$ being
symmetric about $\frac{1}{2}$.

Here is the classification of the finite dimensional irreducible
representations of $Y_n^-$:
\begin{Theorem} 
[{\cite[Theorem 4.8]{M}}]
\label{y_n_minus} 
  The $Y_n^-$-module $L(\bar \mu(u))$ is finite
dimensional if and only if  
\[
\mu_1(-u) \Rightarrow \mu_1(u) \rightarrow \mu_3(u) \rightarrow \dots \rightarrow \mu_{n-1}(u).
\]
\end{Theorem}

To obtain a more combinatorial description of the finite dimensional
irreducible representations of $Y_n^-$, 
we prove the following lemmas.

\begin{Lemma} \label{lRightarrow}
If $\mu(u) \in 1+ u^{-1} \C[[u^{-1}]]$ then
$\mu(-u) \Rightarrow \mu(u) $ if and only if there exists 
$\gamma(u) \in 1 + u^{-2}\C[[u^{-2}]]$ such that
$\gamma(u) \mu(u) = (1-a_1 u^{-1})(1-a_2 u^{-1}) \dots (1-a_{2k}u^{-1})$ 
where 
\begin{equation} \label{eq_pair}
a_{2i-1} + a_{2i} \ge 0 \text{ for } i=1, \dots, k.
\end{equation}
\end{Lemma}
\begin{proof}
Assume $\mu(-u) \Rightarrow \mu(u)$, so there exists
a monic polynomial $P(u)$ of even degree
so that $P(u) = P(1-u)$ and
\[
  \frac{\mu(-u)}{\mu(u)} = \frac{P(u+1)}{P(u)}.
\]
Let $2k$ be the degree of $P(u)$,
and let
\[ 
  \gamma(u) = \frac{P(u)u^{-2k}}{\mu(u)}.
\]
So $\gamma(u) \mu(-u) = P(u+1)u^{-2k}$ and $\gamma(u) \mu(u) = P(u)u^{-2k}$.
Since the roots of $P(u)$ are symmetric about $\frac{1}{2}$, we can write
\[
  \gamma(u) \mu(u) = (1-b_1 u^{-1})(1-(1-b_1)u^{-1}) \dots
   (1-b_k u^{-1}) (1-(1-b_k)u^{-1}).
\]
Now it is clear that $\gamma(u)\mu(u)$ satisfies \eqref{eq_pair},
so it remains to see that
$\gamma(u) \in \C[[u^{-2}]]$.  Note that the roots of 
$P(u+1) = u^{2k} \gamma(u) \mu(-u)$ are $b_1-1, -b_1, \dots, b_k-1, -b_k$.
Now since these are also the roots of $P(-u) = u^{2k} \gamma(-u) \mu(-u)$, 
we have that
$\gamma(-u)\mu(-u) = \gamma(u) \mu (-u)$,  so
$\gamma(-u) = \gamma(u)$, and thus $\gamma(u) \in \C[[u^{-2}]]$.

Conversely, we now assume that there exists $\gamma(u) \in 1+ u^{-2}\C[[u^{-2}]]$ 
such that
$\gamma(u) \mu(u) = (1-a_1 u^{-1})(1-a_2 u^{-1}) \dots (1-a_{2k}u^{-1})$, 
where $a_{2i-1} + a_{2i} \ge 0$ for $i=1, \dots, k$.
Let 
$P_i(u) = (u+a_{2i-1}-1) (u+a_{2i-1}-2) \dots (u-a_{2i}+1)$, 
and let
$Q_i(u) = (u+a_{2i}-1) (u+a_{2i}-2) \dots (u-a_{2i-1}+1)$.
Now it is 
the case that
\[
  \frac{\mu(-u)}{\mu(u)} = 
   \frac{(u+a_1) P_1(u)(u+a_2)Q_1(u) \dots (u+a_{2k-1}) P_k(u) (u+a_{2k})Q_k(u)}
    { P_1(u)(u-a_2) Q_1(u) (u-a_1) \dots P_k(u) (u-a_{2k})  Q_k(u)(u-a_{2k-1})},
\] 
so 
$P(u) = P_1(u)(u-a_2) Q_1(u) (u-a_1) \dots P_k(u) (u-a_{2k})  Q_k(u)(u-a_{2k-1})$
is the unique monic polynomial of even degree such that $P(u) = P(1-u)$ and
\[
  \frac{\mu(-u)}{\mu(u)} = \frac{P(u+1)}{P(u)}.
\]
\end{proof}

\begin{Lemma} \label{l_min}
Suppose that $\mu(u) \in 1 + u^{-1} \C[u^{-1}]$, $\mu(-u) \Rightarrow \mu(u)$, and there exists
$\gamma(u) \in 1 + u^{-2}\C[[u^{-2}]]$ such that
$\gamma(u) \mu(u) = (1-a_1 u^{-1})(1-a_2 u^{-1}) \dots (1-a_{2k}u^{-1})$ 
where $a_{2i-1} + a_{2i} \ge 0$ for $i=1, \dots, k$.
Then there exists $\gamma'(u) \in 1 + u^{-2}\C[[u^{-2}]]$ such that after re-indexing
$(a_1, \dots, a_{2k})$ we can write
$\gamma'(u) \mu(u) = (1-a_1 u^{-1})(1-a_2 u^{-1}) \dots (1-a_{2m}u^{-1})$
where 
$m \le k$, 
for each $i \ne j \in \{1, \dots, 2 m\}$ we have that $a_i \ne - a_j$,
and
$(a_1, \dots, a_{2m})$ satisfies
$a_{2i-1} + a_{2i} \ge 0$ for $i=1, \dots, m$.
\end{Lemma}
\begin{proof}
We proceed by induction on $k$, and assume that $a_i = -a_j$ for some
$i \ne j \in \{1, \dots, 2k\}$.
After re-indexing we may assume that $a_1 = -a_2$ or $a_1 = -a_3$.
If $a_1 = -a_2$ then
\[
\frac{\gamma(u)\mu(u)}{1-a_1^2 u^{-2}} = (1-a_3 u^{-1}) \dots (1-a_{2k}u^{-1})
\]
satisfies the hypotheses of the lemma, so the lemma follows by induction.
If $a_1 = -a_3$ then we have that $a_2 + a_4 = a_1+a_2 +a_3+a_4 \ge 0$,
so 
\[
\frac{\gamma(u) \mu(u)}{1-a_1^2 u^{-2}} = (1-a_2 u^{-1}) (1-a_4 u^{-1}) \dots (1-a_{2k}u^{-1})
\]
satisfies the hypotheses of the lemma, so the lemma follows by induction.
\end{proof}

\begin{Lemma} \label{lanyevenpoly}
Let $\mu(u) \in 1+ u^{-1} \C[[u^{-1}]]$.
If $\mu(-u) \Rightarrow \mu(u)$ and $\gamma(u) \in 1+u^{-2}\C[[u^{-2}]]$ is
such that $\gamma(u) \mu(u) \in \C[u^{-1}]$ then we can write
$\gamma(u) \mu(u) = (1-a_1 u^{-1})(1-a_2 u^{-1}) \dots (1-a_{2k}u^{-1})$
so that
$a_{2i-1} + a_{2i} \ge 0$ for $i=1, \dots, k$.
\end{Lemma}
\begin{proof}
By Lemmas \ref{lRightarrow} and \ref{l_min}
there exists $\gamma'(u) \in 1+u^{-2} \C[[u^{-2}]]$ such that
$\gamma'(u)\mu(u) = (1-b_1 u^{-1}) \dots (1-b_{2m} u^{-1})$ so that
$b_i \ne -b_j$ for all $i \ne j \in \{1, \dots, 2m\}$ and so that 
$b_{2i-1} + b_{2i} \ge 0$ for $i=1, \dots, m$.
Write 
\[
  \gamma(u)\mu(u)=
   (1-a_1 u^{-1}) \dots (1-a_p u^{-1}) (1- a_{p+1}^2 u^{-2}) \dots (1-a_q^2 u^{-2}) 
\]
such that 
$a_i \ne -a_j$ for all $i \ne j \in \{1, \dots, p\}$.
Thus
\[
 \frac{\gamma(u)}{\gamma'(u)}
 = \frac{(1-a_1 u^{-1}) \dots (1-a_p u^{-1}) (1- a_{p+1}^2 u^{-2}) \dots (1-a_q^2 u^{-2}) }
{(1-b_1 u^{-1}) \dots (1-b_{2m} u^{-1})}
\]
which implies
\begin{align*}
    &(1-a_1 u^{-1}) \dots (1-a_p u^{-1}) (1+b_1 u^{-1}) \dots (1+b_{2m} u^{-1})  \\
   & \quad  = \frac{\gamma(u) (1-b_1^2 u^{-2})\dots (1-b^2_{2m} u^{-2})}
     {\gamma'(u)  (1- a_{p+1}^2 u^{-2}) \dots (1-a_q^2 u^{-2})} \in \C[[u^{-2}]],
\end{align*}
and thus $p=2m$ and after re-indexing we must have that 
$a_i = b_i$ for all $i \in \{1, \dots, 2m\}$.
\end{proof}

\begin{Lemma} \label{l_cond}
Let $\mu_1(u), \mu_2(u), \dots, \mu_m(u) \in 
1 + u^{-1}\C[[u^{-1}]]$.
Suppose $\mu_1(-u) \Rightarrow \mu_1(u) \rightarrow \mu_2(u)
\rightarrow \dots \rightarrow \mu_m(u)$.  Then there exists
$\gamma(u) \in 1 + u^{-2}\C[[u^{-2}]]$ such that
$\gamma(u) \mu_i(u) \in \C[u^{-1}]$
for $i = 1, \dots, m$.
\end{Lemma}

\begin{proof}
By Lemma \ref{lcompat} there exists
$\upsilon(u) \in 1+u^{-1}\C[[u^{-1}]]$ such that
$\upsilon(u) \mu_i(u) \in \C[u^{-1}]$.
So we can write $\upsilon(u) \mu_1(u) = (1+b_1 u^{-1}) \dots (1+b_s u^{-1})$.
Let $\upsilon'(u) = \upsilon(u) (1-b_1 u^{-1}) \dots (1-b_s u^{-1})$,
so $\upsilon'(u) \mu_1(u) \in \C[u^{-2}]$, and 
$\upsilon'(u) \mu_i(u) \in \C[u^{-1}]$ for $i=1, \dots, m$.
By Lemma \ref{lRightarrow} there exists
$\eta(u) \in 1+u^{-2} \C[[u^{-2}]]$ such that
$\eta(u) \mu_1(u) \in \C[u^{-1}]$.
Let $\gamma(u) = \eta(u) \upsilon'(u) \mu_1(u)$.  
Now $\gamma(u) \in 1 + u^{-2} \C[[u^{-2}]]$
and $\gamma(u) \mu_i(u) \in \C[u^{-1}]$ for $i=1, \dots, m$.
\end{proof}

This lemma is key to giving a more combinatorial description 
of the finite dimensional irreducible representations of $Y$ which is done (in
the context of representations of finite $W$-algebras) in $\S$\ref{section4} below.

Next we turn our attention to the classification of finite dimensional irreducible
representations of $Y_n^+$ when $n$ is even.
The $n=2$ case needs to be treated separately from the 
$n>2$ cases.
\begin{Theorem} 
[{\cite[Proposition 5.3]{M}}]
\label{y_2} 
The $Y_2^+$-module $L((\mu_1(u)))$ is finite dimensional if and
only if
there exists $\gamma(u) \in 1 + u^{-2}\C[[u^{-2}]]$ such that
\[
  (1+\frac{1}{2} u^{-1}) \gamma(u) \mu_1(u) = 
  (1 - a_1 u^{-1}) (1-a_2 u^{-1}) \dots (1-a_{2k+1} u^{-1}) 
\]
where $a_{2i-1} + a_{2i} \ge 0$ for $i=1, \dots, k$. 
\end{Theorem}

We need a slight generalization of this theorem.
\begin{Lemma} \label{molev5.3}
   Let $\mu_1(u) \in 1 + u^{-1} \C[[u^{-1}]]$.  
If the $Y_2^+$-module $L((\mu_1(u)))$ is finite dimensional and
$\gamma(u) \in 1 +   u^{-2}\C[[u^{-2}]]$ is such that
$(1+\frac{1}{2} u^{-1}) \gamma(u) \mu_1(u) \in \C[u^{-1}]$ then we can write
\[
  (1+\frac{1}{2} u^{-1}) \gamma(u) \mu_1(u) = 
  (1 - a_1 u^{-1}) (1-a_2 u^{-1}) \dots (1-a_{2k+1} u^{-1}),
\]
where
$a_{2i-1} + a_{2i} \ge 0$ for $i=1, \dots, k$
\end{Lemma}
\begin{proof}
Suppose that such a $\gamma(u)$ exists.
By \cite[Theorem 5.4]{M}
$L((\mu_1(u)))$ is finite dimensional if and only if there exists a monic
polynomial
$P(u) \in \C[u]$ with $P(u) = P(-u+1)$ and $c \in \C$ such that
$P(-c) \ne 0$ and
\[ \frac{\mu_1(-u)}{\mu_1(u)}
   = 
   \frac{P(u+1) (u+c)(2u+1)}{P(u)(u-c)(2u-1)}.
\]
Let $\lambda(u) = \mu_1(u) (1+c u^{-1}) (1 + \frac{1}{2} u^{-1})$.
Thus we have that
   $\lambda(-u) \Rightarrow \lambda(u)$, and  
   since $\gamma(u) \lambda(u) = 
(1 - a_1 u^{-1}) (1-a_2 u^{-1}) \dots (1-a_{2k+1} u^{-1})
(1+c u^{-1})$,
by Lemma \ref{lanyevenpoly} 
after re-indexing we have that $a_{2i-1} + a_{2 i} \ge 0$ 
for $i = 1, \dots, k$.
\end{proof}

Next we will give the classification of  
finite dimensional irreducible $Y_n^+$-modules
for $n$ even, $n > 2$.
This depends on a certain $Y_n^+$ automorphism $\psi$:
\begin{equation} \label{eq_psi}
    \psi: Y_n^+ \to Y_n^+, \quad S_{i,j}(u) \mapsto S_{i',j'}(u),
\end{equation}
where
$i' = i$ if $i \ne \pm 1$, and $i' = -i$ if $i=\pm 1$.

If $L$ is a $Y^+_n$-module, we let $L^\sharp$ denote the module created
by twisting with $\psi$, that is, if $v \in L$, $y \in Y^+_n$, then
$L^\sharp$ is the module created by the action $y.v = \psi(y)v$, 
where $\psi(y)v$ denotes the action given by $L$.
Of course, if $L(\bar \mu(u))$ is a
finite dimensional $Y$-module, then so is
$L(\bar \mu(u))^\sharp$, and by
Theorem \ref{th_tyhw2}
$L(\bar \mu(u))^\sharp$ is another highest 
weight module.
To determine which highest weight module, 
we need to define the notation of a {\em $\sharp'$-special} element of a list of 
complex numbers.
A list $(a_1, a_2, \dots, a_{2k+1})$ 
of complex numbers can be indexed 
so that the following condition is satisfied:
\begin{align} \label{eq_sharp}
  &\text{for every $i=1, \dots, k$ we have:} \notag \\
  & \text{if the set $\set{a_p + a_q }{ 2 i -1 \le p < q \le 2k+1} \cap \Z_{\ge 0}$
      is non-empty} \\
  & \text{then $a_{2i-1} + a_{2i}$ is its minimal element.} \notag
\end{align}
For an element $a$ in a list $(a_1, a_2, \dots, a_{2k+1})$ of complex numbers,
we say that $a$ is a {\em $\sharp'$-special} element of $(a_1, a_2, \dots, a_{2k+1})$ if
$a=a_{2k+1}$ when $(a_1, \dots, a_{2k+1})$ is indexed so that \eqref{eq_sharp} holds.

Recall the definition of the a $\sharp$-special element of 
a list of complex numbers
from the introduction.
The following lemma shows that the concepts of the $\sharp$-special and $\sharp'$-special elements of a list are nearly identical.
\begin{Lemma} \label{l_sharp_max}
Let $(a_1, \dots, a_{2k+1})$ be a list of complex numbers.
If the $\sharp$-special element of the list
$(a_1+1/2, \dots, a_{2k+1} +1/2)$ is defined then
$a_{2k+1}+1/2$ is the $\sharp$-special element of this list
if and only if 
$a_{2k+1}$ is the $\sharp'$-special element of the list
$(a_1, \dots, a_{2k+1})$.
In particular, the $\sharp'$-special element 
is unique in these
circumstances.
\end{Lemma}
\begin{proof}
We proceed by induction on $k$, the case $k=0$ being clear.
Let $(a_1, \dots, a_{2k+1})$ be a list for which $a_{2i-1}+a_{2i} \ge 0$ for
$i =1, \dots, k$, and for which \eqref{eq_sharp} holds.
Let $(b_1, \dots, b_{2k+1})$ be a re-indexing of $(a_1, \dots, a_{2k+1})$ such that
 $b_{2i-1}+b_{2i} \ge 0$ for $i =1, \dots, k$.
Assume that $b_{2k+1} \ne a_{2k+1}$.  Then after re-indexing we may assume that
 $b_1 =a_{2k+1}$.  Let $i$ be such that  $a_{i} = b_2$.
We assume that $i$ is odd, as the case that $i$ is even is proved similarly.
Since $(a_1, \dots, a_{2k+1})$ satisfies \eqref{eq_sharp}, we have that
$a_i+a_{i+1} \le a_i + a_{2k+1}$, so
$a_{i+1} \le a_{2k+1}$.  
If $k=1$ then we must have that $a_{i+1} = b_3$, so the lemma holds in this case.  
If $k > 1$ then
after re-indexing we may assume that
 $a_{i+1} = b_3$,
so
$a_{2k+1} + b_4 \ge 0$.
Now we have that  the lists $(a_1, \dots, a_{i-1}, a_{i+2}, \dots, a_{2k+1})$
and 
$(a_{2k+1}, b_4, \dots, b_{2k +1} )$ also satisfy the hypotheses of the lemma,
so by induction $b_{2k+1} \le a_{2k+1}$.
\end{proof}

Suppose $\mu(u) \in 1+ u^{-1} \C[[u^{-1}]]$ is such that there exists
$\gamma(u) \in 1+u^{-2}\C[[u^{-2}]]$ such that
$(1+\frac{1}{2} u^{-1}) \gamma(u) \mu(u) = 
(1-a_1 u^{-1}) \dots (1-a_{2k+1} u^{-1})$, where
$a_{2i-1} + a_{2i} \ge 0$ for $i=1, \dots, k$ and 
$(a_1, \dots, a_{2k+1})$
satisfies \eqref{eq_sharp}.
If these conditions are met then we say that {\em $\mu^\sharp(u)$ is well-defined}.
Now we define 
\begin{equation} \label{eq_musharp}
  \mu^\sharp(u) = \gamma(u)^{-1} (1+\frac{1}{2} u^{-1})
   (1-a_1 u^{-1}) \dots (1-a_{2k} u^{-1}) (1+(1+a_{2k+1})u^{-1})
\end{equation}

\begin{Lemma}
  The definition of $\mu^\sharp(u)$ is well-defined, that is, it does not depend on $\gamma(u)$.
\end{Lemma}
\begin{proof}
  First we make the following observation.
If $(a_1, \dots, a_{2k+1})$ satisfies
\eqref{eq_sharp} and $a_{2i-1}+a_{2i} \ge 0$ for
$i=1, \dots, k$ then for any $a \in \C$ 
the list
$(a, -a, a_1, \dots, a_{2k+1})$
also satisfies \eqref{eq_sharp}.
Note this also implies that if $a$ and $-a$ both occur
in $(a_1, \dots, a_{2k+1})$ then the $\sharp'$-special
element of the list $(a_1, \dots, a_{2k+1})$ with one occurrence of $a$
and $-a$ removed is also $a_{2k+1}$.

Now suppose that for 
$\gamma(u) \in 1+u^{-2} \C[[u^{-2}]]$ that
$(1+\frac{1}{2} u^{-1}) \gamma(u) \mu(u) = 
(1-a_1 u^{-1}) \dots (1-a_{2k+1} u^{-1})$, where
$a_{2i-i} + a_{2i} \ge 0$ for $i=1, \dots, k$ and 
$(a_1, \dots, a_{2k+1})$
satisfies \eqref{eq_sharp}.
Also suppose for some $\gamma'(u) \in 1+u^{-2} \C[[u^{-2}]]$ that
$(1+\frac{1}{2} u^{-1}) \gamma'(u) \mu(u) = 
(1-b_1 u^{-1}) \dots (1-b_{2k'+1} u^{-1})$, where
$b_{2i-i} + b_{2i} \ge 0$ for $i=1, \dots, k'$ 
and 
$(b_1, \dots, b_{2k'+1})$
satisfies \eqref{eq_sharp}.

By re-indexing we may write
\[
(1-a_1 u^{-1}) \dots (1-a_{2k+1} u^{-1})
 = (1-a_1 u^{-1}) \dots (1-a_p u^{-1}) 
  (1-a_{p+1}^2 u^{-2}) \dots (1-a_q^2 u^{-2})
\]
and
\[
(1-b_1 u^{-1}) \dots (1-b_{2k'+1} u^{-1})
 = (1-b_1 u^{-1}) \dots (1-b_{p'} u^{-1}) 
  (1-b_{p'+1}^2 u^{-2}) \dots (1-b_{q'}^2 u^{-2}),
\]
where
$a_i \ne a_j$ for all $i \ne j \in \{1, \dots, p\}$ and
$b_i \ne b_j$ for all $i \ne j \in \{1, \dots, p'\}$.
So
\[
\frac{(1-a_1 u^{-1}) \dots (1-a_p u^{-1}) 
  (1-a_{p+1}^2 u^{-2}) \dots (1-a_q^2 u^{-2})}
{(1-b_1 u^{-1}) \dots (1-b_{p'} u^{-1}) 
  (1-b_{{p'}+1}^2 u^{-2}) \dots (1-b_{q'}^2 u^{-2})}
 = \frac{\gamma(u)}{\gamma'(u)},
\]
so
\begin{align*}
&(1-a_1 u^{-1}) \dots (1-a_p u^{-1}) 
(1+b_1 u^{-1}) \dots (1+b_{p'} u^{-1})  \\
   & \quad =
  \frac{\gamma(u) (1-b_1^2 u^{-2}) \dots (1-b_{q'}^2 u^{-2})}
    { \gamma'(u) (1-a_{p+1}^2 u^{-2}) \dots (1-a_q^2 u^{-2})}
   \in \C[[u^{-2}]].
\end{align*}
Thus $p=p'$, and after re-indexing,
$a_i=b_i$ for $i=1, \dots, p$.
Now the lemma follows from the above observation.
\end{proof}

The following theorem is contained in the proof of
\cite[Theorem 5.9]{M}.
\begin{Theorem}
\label{th_sharp}
Let $\mu_1(u), \mu_3(u), \dots, \mu_{n-1}(u) \in 1+ u^{-1} \C[[u^{-1}]]$
where $\mu_1^\sharp(u)$ is well-defined.
Then
$L((\mu_1(u), \mu_3(u), \dots, \mu_{n-1}(u)))^\sharp =
L((\mu_1^\sharp(u), \mu_3(u), \dots, \mu_{n-1}(u)))$.
\end{Theorem}

We will also need the following lemma.
\begin{Lemma} \label{l_sharp0}
If $(a_1, \dots, a_{2k+1})$ satisfies
$a_{2i-1} + a_{2i} \ge 0$ for $i=1, \dots, k$ 
and $a_{2k+1}$ is 
the $\sharp'$-special element of 
$(a_1, \dots, a_{2k+1})$ 
then
$-1-a_{2k+1}$ is the $\sharp'$-special element of
$(a_1, \dots, a_{2k}, -1-a_{2k+1})$.
\end{Lemma}
\begin{proof}
Theorem \ref{y_2} and Lemma \ref{molev5.3} 
imply that
the $Y_2^+$-module 
$L((\mu_1(u)))$, where
$\mu_1(u) = (1-a_1 u^{-1})  \dots (1-a_{2k+1} u^{-1}) (1+\frac{1}{2} u^{-1})^{-1})$,
is finite dimensional, and by
\eqref{eq_musharp} and Theorem \ref{th_sharp} 
\[
L((\mu_1(u)))^\sharp 
 = L(((1-a_1 u^{-1}) \dots (1-a_{2k} u^{-1}) (1+(1+a_{2k+1}) u^{-1}) 
   (1+\frac{1}{2} u^{-1})^{-1})).
\]
Since $\psi$ from \eqref{eq_psi} is an involution, we must have that
\[
  (L((\mu_1(u))) ^ \sharp)^\sharp
   = L(((1-a_1 u^{-1})  \dots (1-a_{2k+1} u^{-1}) (1+\frac{1}{2} u^{-1})^{-1})).
\]
Now suppose $\sharp'$-special element of 
$(a_1, \dots, a_{2k}, -1-a_{2k+1})$ is $a_j$ for some
$j \in \{1, \dots, 2k\}$.
So by Theorem \ref{th_sharp} 
$(a_1, \dots, a_{2k+1}) = (a_1, \dots, a_{j-1}, -1-a_j,
a_{j+1}, \dots, a_{2k}, -1-a_{2k+1})$, so we must have that
$(a_j, a_{2k+1}) = ( -1-a_j, -1-a_{2k+1})$.  Since
$a_j \ne -1-a_{2k+1}$, we must have that $a_{2k+1} = -1-a_{2k+1}$, which implies that
$a_{2k+1} = -\frac{1}{2} = -1-a_{2k+1}$, so by Lemma \ref{l_sharp_max} 
the $\sharp'$-special element
of $(a_1, \dots, a_{2k}, -1-a_{2k+1})$ is in fact $-1-a_{2k+1}$.
\end{proof}

Here is the classification of the finite dimensional 
irreducible representations of $Y_n^+$
for even $n >2$.
\begin{Theorem}[{\cite[Theorem 5.9]{M}}]
\label{y_n_plus}
Let $n >2$ be even.  Then the
$Y_n^+$-module 
$L(\bar \mu(u))$ is finite dimensional
if and only if $\mu_1^\sharp(u)$ is well defined 
and any of the following four conditions holds:
\begin{itemize}
\item[(i)]
   $ \mu_1(-u) \Rightarrow \mu_1(u) \rightarrow \mu_3(u) \rightarrow \dots
     \rightarrow \mu_{n-1}(u)$,
\item[(ii)]
   $ \frac{2u-1}{2u+1} \mu_1(-u) \Rightarrow \mu_1(u) \rightarrow \mu_3(u) \rightarrow \dots
     \rightarrow \mu_{n-1}(u)
$,
\item[(iii)] 
   $ \mu_1^\sharp(-u) \Rightarrow \mu_1^\sharp(u) \rightarrow \mu_3(u) \rightarrow \dots
     \rightarrow \mu_{n-1}(u)$,
\item[(iv)]
   $ \frac{2u-1}{2u+1} \mu_1^\sharp(-u) \Rightarrow \mu_1^\sharp(u) \rightarrow \mu_3(u) \rightarrow \dots
     \rightarrow \mu_{n-1}(u)$.
\end{itemize}
\end{Theorem}

In order to give a more combinatorial description 
of this classification we need the following four lemmas.  In each
of the lemmas we assume for some $\gamma(u) \in 1+u^{-2} \C[[u^{-2}]]$ that
\[
 \mu(u) = \gamma(u)^{-1} 
(1+\frac{1}{2} u^{-1})^{-1}
(1-a_1 u^{-1}) \dots (1-a_{2k+1} u^{-1}) 
\]
where 
$a_{2k+1}$ is the $\sharp'$-special element of 
$(a_1, \dots, a_{2k+1})$
and
$a_{2i-1} + a_{2i} \ge 0$ for $i=1, \dots, k$.

\begin{Lemma} \label{l_sharp1}
Let $\mu(u)$ be as above.  
Then $\mu(-u) \Rightarrow \mu(u)$ 
if and only if 
  $a_{2 k + 1} \ge - \frac{1}{2}$.
\end{Lemma}
\begin{proof}
Suppose $\mu(-u) \Rightarrow \mu(u)$.
Since $(1-\frac{1}{4} u^{-2}) \gamma(u) \mu(u) \in \C[u^{-1}]$,
by Lemma \ref{lanyevenpoly} the list $(a_1, \dots, a_{2k+1}, \frac{1}{2})$ 
can be re-indexed as
$(b_1, \dots, b_{2k+2})$ where
$b_{2i-1} + b_{2i} \ge 0$ for $i=1, \dots, k+1$.
So $b_j = \frac{1}{2}$ for some $j$. 
We assume that $j$ is odd, the proof when $j$ is even is similar.
By Lemma \ref{l_sharp_max} $b_{j+1} \le a_{2k+1}$, so
$0 \le b_{j+1} + \frac{1}{2} \le a_{2k+1} + \frac{1}{2}$.

To prove the converse note that
\[
\gamma(u)(1-\frac{1}{4} u^{-2}) 
\mu(u) = (1-a_1 u^{-1}) \dots (1-a_{2k+1} u^{-1}) (1-\frac{1}{2} u^{-1}),
\]
then apply Lemma \ref{lRightarrow}.

\end{proof}

\begin{Lemma} \label{l_sharp2}
  Let $\mu(u)$ be as above.  Then 
$\frac{2u-1}{2u+1} \mu(-u) \Rightarrow \mu(u)$ 
if and only if 
  $a_{2 k + 1} \ge 0$.
\end{Lemma}
\begin{proof}
Note that 
\begin{equation} \label{eq_l_sharp2}
\frac{2u-1}{2u+1} \mu(-u) \Rightarrow \mu(u) 
\text{ if and only if }
(1-\frac{1}{2} u^{-1}) \mu(-u) \Rightarrow (1+\frac{1}{2} u^{-1})\mu(u).
\end{equation}

Suppose $\frac{2u-1}{2u+1} \mu(-u) \Rightarrow \mu(u)$.
So by Lemma \ref{lanyevenpoly} and \eqref{eq_l_sharp2},
the set $(a_1, \dots, a_{2k+1}, 0)$ 
can be re-indexed as
$(b_1, \dots, b_{2k+2})$ where
$b_{2i-1} + b_{2i} \ge 0$ for $i=1, \dots, k+1$.
So  $b_j = 0$ for some $j$. 
We assume that $j$ is odd, the proof when $j$ is even is similar. 
By Lemma \ref{l_sharp_max} $b_{j+1} \le a_{2k+1}$, so
$0 \le b_{j+1} \le a_{2k+1}$.

The converse follows immediately from Lemma \ref{lRightarrow} 
and \eqref{eq_l_sharp2}
since
\[
\gamma(u) (1+\frac{1}{2} u^{-1}) 
\mu(u) = (1-a_1 u^{-1}) \dots (1-a_{2k+1} u^{-1})(1-0 u^{-1}).
\]
\end{proof}

\begin{Lemma} \label{l_sharp3}
   Let $\mu(u)$ be as above.  Then $\mu^\sharp(-u) \Rightarrow \mu^\sharp(u)$ 
if and only if 
  $a_{2 k + 1} \le -\frac{1}{2}$.
\end{Lemma}
\begin{proof}
This follows from \eqref{eq_musharp}, Lemma \ref{l_sharp0}, and Lemma \ref{l_sharp1}.
\end{proof}

\begin{Lemma} \label{l_sharp4}
  Let $\mu(u)$ be as above.  Then $\frac{2u-1}{2u+1} \mu^\sharp(-u) \Rightarrow \mu^\sharp(u)$ 
    if and only if 
  $ a_{2 k + 1} \le -1$.
\end{Lemma}
\begin{proof}
This follows from \eqref{eq_musharp}, Lemma \ref{l_sharp0}, and Lemma \ref{l_sharp2}.
\end{proof}

Next we give the classification of the finite dimensional irreducible 
$Y_n^+$-modules for $n$ odd.
Note that for a highest weight representation 
of highest weight
$\bar \mu(u) \in (1+u^{-1} \C[[u^{-1}]])^{(n+1)/2}$, 
by the relation
\eqref{eq_y_rel2}, we must have that $\mu_0(u) \in 1+u^{-2}\C[[u^{-2}]]$.

\begin{Theorem} 
[{\cite[Theorem 6.7]{M}}]
\label{y_n_odd} 
Assume that $n \in \Z_{> 0}$ is odd.  Then the $Y_n^+$-module
$L(\bar \mu(u))$ is finite dimensional 
if and only if either one of the the following two 
conditions holds:
\begin{itemize}
\item[(i)]
$
\mu_0(u) \rightarrow \mu_2(u) \rightarrow \dots \rightarrow \mu_{n-1}(u),
$
\item [(ii)]
$
\frac{2u}{2u+1} \mu_0(u) \rightarrow \mu_2(u) \rightarrow \dots \rightarrow \mu_{n-1}(u).
$
\end{itemize}
\end{Theorem}

\section {Proof of the classification theorem} \label{section4}

In this section we prove Theorem \ref{t_class} on a case by case basis.
First we recall from the introduction that we associate 
$A = (a_{i,j})_{i \in I_n, j \in I_l} \in \Tab_{n,l}$ to an
irreducible
highest weight $Y$-module $L(A)$ with highest weight vector $v$
by declaring that 
\[
  (u - \frac{i}{2})^l S_{i,i}(u-\frac{i}{2}) v = (u+a_{i,1-l})(u+a_{i,3-l}) \dots (u+a_{i,l-1}) v
\]
   if $l$ is even and $i \ge 0$, or
\begin{align*}
  (&u - \frac{i}{2})^{l-1} (u+\frac{\phi-i}{2}) S_{i,i}(u-\frac{i}{2}) v  \\
   &= 
   (u+a_{i,1-l})(u+a_{i,3-l}) \dots (u+a_{i,-2})(u+a_{i,0} + 
\delta_{i,0} /2) (u+a_{i,2}) \dots (u+a_{i,l-1}) v
\end{align*}
if $l$ is odd and $i \ge 0$.
In other words, 
this means 
if 
\begin{equation} \label{lofa}
  L(A) =  L(\bar \mu(u))
\end{equation}
and  $l$ is even then
\[
  \mu_i(u) = (1+c_{i,1-l} u^{-1}) (1+c_{i,3-l} u^{-1}) \dots
  (1+c_{i,l-1} u^{-1})
\]
where $c_{i,j} = a_{i,j} + \frac{i}{2}$ for 
$i \in \mathcal{I}_n \cap \Z_{\ge 0}, j \in \mathcal{I}_l$.
If $l$ is odd then this means that 
\[
  \mu_i(u) = (1+\frac{\phi}{2} u^{-1})^{-1}
   (1+c_{i,1-l} u^{-1}) (1+c_{i,3-l} u^{-1}) \dots
  (1+c_{i,l-1} u^{-1})
\]
where $c_{i,j} = a_{i,j} + \frac{i+\delta_{i,0} \delta_{j,0}}{2}$ for 
$i \in \mathcal{I}_n \cap \Z_{\ge 0}, j \in \mathcal{I}_l$.

\begin{Lemma}
Theorem \ref{t_class} holds in the case that $\phi = -, \eps = +$, $n$ is even,
 and $l$ is even.
\end{Lemma}
\begin{proof}
In this case, by \eqref{ker2}, an irreducible highest weight
$Y_n^-$-module  $L(\bar \mu(u))$
factors through
$\kappa_l$ 
if
$\mu_i(u)$ is a polynomial of degree $l$ or less
for all $i \in \{1, 3, \dots, n-1\}$. 
Furthermore, if $L(\bar \mu(u))$ 
is finite dimensional
then 
by Theorem \ref{y_n_minus},
Lemma \ref{lpoly}, and
Lemma \ref{lanyevenpoly}
we can write 
$\mu_i(u) = (1+c_{i,1-l} u^{-1}) (1+c_{i,3-l} u^{-1}) \dots 
(1+c_{i,l-1} u^{-1})$ 
for $i \in \{1, 3, \dots, n-1\}$ 
such that
$c_{1,j} + c_{1,-j} \le 0$  for all $j \in \mathcal{I}_l$ and
$c_{i,j} \ge c_{i+2,j}$ for all $j \in \mathcal{I}_l$,
$i \in \{1, \dots, n-3\}$.
Now associate to this data the skew-symmetric $n \times l$ tableaux 
$A=(a_{i,j})_{i \in \mathcal{I}_n, j \in \mathcal{I}_j}$
indicated by \eqref{lofa}, that is
$a_{i,j} = c_{i,j} -\frac{i}{2}$ for $i \in \{1, 3, \dots, n-1\}$,
$j \in \mathcal{I}_l$, and
$a_{i,j} = - a_{-i,-j}$ for $i  \in \{1-n, 3-n, \dots, -1\}$, $j \in \mathcal{I}_l$.
Now it is clear that $A \in \Col_{n,l}^+$.

It is also easy to see that given a 
skew-symmetric $n \times l$ tableaux 
$A=(a_{i,j})_{i \in \mathcal{I}_n, j \in \mathcal{I}_j} \in \Col_{n,l}^+$,
that 
$L(A) = L(\bar \mu(u))$, where
$\bar \mu(u)$ is given
by \eqref{lofa},
is finite dimensional
by
Theorem \ref{y_n_minus},
Lemma \ref{larrow}, and
Lemma \ref{lRightarrow}. 
\end{proof}

\begin{Lemma}
Theorem \ref{t_class} holds in the case that $\phi = -, \eps = -$, $n$ is even,
and $l$ is odd.
\end{Lemma}
\begin{proof}
In this case, by \eqref{ker2}, an irreducible highest weight
$Y_n^-$-module  $L(\bar \mu(u))$
factors through
$\kappa_l$ 
if
$(1-\frac{1}{2}u^{-1}) \mu_i(u)$ is a polynomial of degree at most $l$ 
for all $i \in \{1, 3, \dots, n-1\}$.  
If $L(\bar \mu(u))$ is finite dimensional
then 
by Theorem \ref{y_n_minus},
Lemma \ref{lpoly}, and
Lemma \ref{lanyevenpoly}
for $i \in \{1, 3, \dots, n-1\}$
we can write 
\[
  (1-\frac{1}{4} u^{-2}) \mu_i(u) = 
   (1+c_{i,1-l} u^{-1})(1+c_{i,3-l}u^{-1}) \dots (1+c_{i,l-1} u^{-1})
   (1+\frac{1}{2} u^{-1})
\]
such that
$c_{1,0} \le - \frac{1}{2}$, 
$c_{1,j} + c_{1,-j} \le 0$  for $0 \ne j \in \mathcal{I}_l$ and
$c_{i,j} \ge c_{i+2,j}$ for $j \in \mathcal{I}_l$,
$i \in \{1, \dots, n-3\}$.
Now associate to this data the skew-symmetric $n \times l$ tableaux 
$A=(a_{i,j})_{i \in \mathcal{I}_n, j \in \mathcal{I}_j}$
indicated by \eqref{lofa}, that is
$a_{i,j} = c_{i,j} -\frac{i}{2}$ for $i \in \{1, 3, \dots, n-1\}$,
$j \in \mathcal{I}_l$, and
$a_{i,j} = - a_{-i,-j}$ for $i  \in \{1-n, 3-n, \dots, -1\}$, $j \in \mathcal{I}_l$.
Now it is clear that $A \in \Col_{n,l}^-$.

It is also easy to see that given a 
skew-symmetric $n \times l$ tableaux 
$A=(a_{i,j})_{i \in \mathcal{I}_n, j \in \mathcal{I}_j} \in \Col_{n,l}^-$,
that 
$L(A) = L(\bar \mu(u))$, where
$\bar \mu(u)$ is given
by \eqref{lofa},
is finite dimensional
by
Theorem \ref{y_n_minus},
Lemma \ref{larrow}, and
Lemma \ref{lRightarrow}.
\end{proof}

\begin{Lemma}
Theorem \ref{t_class} holds in the case that $\phi = +, \eps = +$, $n=2$,
and $l$ is odd.
\end{Lemma}
\begin{proof}
By \eqref{ker2}, an irreducible highest weight
$Y_2^+$-module  $L((\mu_1(u)))$
factors through
$\kappa_l$ 
if
$(1+\frac{1}{2} u^{-1})
\mu_1(u)$ 
 is a polynomial of degree $l$ or less.
Now if $L((\mu_1(u)))$ is finite dimensional then by Theorem \ref{y_2} and
Lemma \ref{molev5.3}
$(1+\frac{1}{2} u^{-1})
\mu_1(u) = (1+c_{1-l}u^{-1}) (1+c_{3-l} u^{-1}) \dots (1+c_{l-1} u^{-1})$
where 
$c_{j} + c_{-j} \le 0$ for $0 \ne j \in \mathcal{I}_l$.
Now associate to this data the skew-symmetric $2 \times l$ tableaux 
$A=(a_{i,j})_{i \in \mathcal{I}_2, j \in \mathcal{I}_j}$
indicated by \eqref{lofa}, that is
$a_{1,j} = c_{1,j} -\frac{1}{2}$ for 
$j \in \mathcal{I}_l$, and
$a_{-1,j} = - a_{1,-j}$ for $j \in \mathcal{I}_l$.
Now it is clear that $A \in \Col_{2,l}^+$.

It is also easy to see that given a
skew-symmetric
$2 \times l$ tableaux 
$A=(a_{i,j})_{i \in \mathcal{I}_2, j \in \mathcal{I}_j} \in \Col_{2,l}^+$,
that
$L(A) = L((\mu_1(u)))$ where
$\mu_1(u)$ is given by
by \eqref{lofa}
is a finite dimensional
$Y_n^-$-module by Theorem \ref{y_2}.
\end{proof}

\begin{Lemma}
Theorem \ref{t_class} holds in the case that $\phi = +, \eps = +$, $n>2$ is even,
and $l$ is odd.
\end{Lemma}
\begin{proof}
 By \eqref{ker2} the irreducible highest weight
$Y_n^+$-module  $L(\bar \mu(u))$
factors through
$\kappa_l$ 
if
$\mu_i(u) (1+\frac{1}{2} u^{-1})$ 
 is a polynomial of degree $l$ or less
for all $i \in \{1, 3, \dots, n-1\}$. 
So we can write
$\mu_i(u) = (1+c_{i,1-l}u^{-1}) (1+c_{i,3-l} u^{-1}) \dots (1+c_{i, l-1} u^{-1}) 
 (1+\frac{1}{2} u^{-1})^{-1}$ for all $i \in \{1, 3, \dots, n-1\}$.

If $L(\bar \mu(u))$ is finite dimensional
then we need to examine the implications from the four conditions in
Theorem \ref{y_n_plus} separately.

If condition (i) or condition (ii) holds from Theorem \ref{y_n_plus} then
by Lemma \ref{lpoly} 
we can re-index each row of the matrix $(c_{i,j})_{i \in \{1, 3, \dots, n-1\}, j \in \mathcal{I}_l}$
so that
\begin{equation} \label{eq_1}
-c_{1,0} \quad \text{ is the $\sharp'$-special element of}
 \quad (-c_{1,1-l}, -c_{1,3-l}, \dots, -c_{1,l-1})
\end{equation}
and
\begin{equation} \label{eq_2}
c_{i,j} \ge c_{i+2,j}  \quad \text{ for }  \quad j \in \mathcal{I}_l,
i \in \{1, \dots, n-3\}.
\end{equation}

If condition (i)  from Theorem \ref{y_n_plus} holds 
then since 
\[
(1-\frac{1}{4} u^{-2}) \mu_i(u) = 
(1+c_{i,1-l}u^{-1}) (1+c_{i,3-l} u^{-1}) \dots (1+c_{i, l-1} u^{-1}) 
 (1-\frac{1}{2} u^{-1})
\]
for all $i \in \{1, 3, \dots, n-1\}$,
we have by 
Lemmas \ref{lanyevenpoly} and \ref{l_sharp1} 
that can further re-index 
so that \eqref{eq_1} and \eqref{eq_2} still hold, 
$c_{1,j} + c_{1,-j} \le 0$  for $0 \ne j \in \mathcal{I}_l$, and
$c_{1,0} \le \frac{1}{2}$.
Since
$c_{3,0} \le c_{1,0}$  
we now have that
$c_{1,0} + c_{3,0} \le 1$ 

If condition (ii) holds from Theorem \ref{y_n_plus} then 
since
$(1-\frac{1}{2} u^{-1}) \mu_1(-u) \Rightarrow (1+\frac{1}{2} u^{-1}) \mu_1(u)
 = (1+c_{i,1-l}u^{-1}) (1+c_{i,3-l} u^{-1}) \dots (1+c_{i, l-1} u^{-1})$,
by Lemmas \ref{lanyevenpoly} and \ref{l_sharp2} we
 can further re-index so that
\eqref{eq_1} and \eqref{eq_2} still hold, 
$c_{1,j} + c_{1,-j} \le 0$  for $0 \ne j \in \mathcal{I}_l$,
and $c_{1,0} \le 0$.
Since
$c_{3,0} \le c_{1,0}$
we now have that
$c_{1,0} + c_{3,0} \le 0$.

If condition (iii) or condition (iv) holds from Theorem \ref{y_n_plus} then
by Lemma \ref{lpoly}  and \eqref{eq_musharp} 
we can re-index each row of the matrix $(c_{i,j})_{i \in \{1, 3, \dots, n-1\}, j \in \mathcal{I}_l}$
so that
\begin{equation} \label{eq_3}
-c_{1,0} \quad \text{ is the $\sharp'$-special element of}
 \quad (-c_{1,1-l}, -c_{1,3-l}, \dots, -c_{1,l-1}),
\end{equation}
\begin{equation} \label{eq_4}
c_{i,j} \ge c_{i+2,j}  \quad \text{ for }  
\quad 0 \ne j \in \mathcal{I}_l,
i \in \{1, \dots, n-3\},
\end{equation}
\begin{equation} \label{eq_5}
c_{i,0} \ge c_{i+2,0} \quad \text{ for } \quad 
i \in \{3, \dots, n-3\},
\end{equation}
and
\begin{equation} \label{eq_6}
c_{1,0} + c_{3,0} \le 1
\end{equation}
(here \eqref{eq_6} holds since $\mu_1(u)^\sharp \rightarrow \mu_3(u)$).

If condition (iii) holds from Theorem \ref{y_n_plus} then since
$(1-\frac{1}{4} u^{-1}) \mu_1^\sharp(u) = 
(1+c_{i,1-l}u^{-1})  \dots 
(1+(1-c_{1,0}) u^{-1}) \dots
(1+c_{i, l-1} u^{-1}) 
 (1-\frac{1}{2} u^{-1})$,
 by Lemmas \ref{lanyevenpoly} and \ref{l_sharp3} 
we can  further re-index so that
\eqref{eq_3}, \eqref{eq_4}, \eqref{eq_5}, and \eqref{eq_6} still hold, 
$c_{1,j} + c_{1,-j} \le 0$  for $0 \ne j \in \mathcal{I}_l$,
and $c_{1,0} \ge \frac{1}{2}$.
Since
$-c_{3,0} \ge c_{1,0} -1$ we now have that
$c_{1,0} - c_{3,0} \ge c_{1,0} + c_{1,0} - 1 \ge 0$.

If condition (iv) holds from Theorem \ref{y_n_plus} then
since
$(1-\frac{1}{2} u^{-1}) \mu_1^\sharp(-u) \Rightarrow (1+\frac{1}{2} u^{-1}) \mu_1^\sharp(u)
 = (1+c_{i,1-l}u^{-1}) \dots (1+(1-c_{1,0}) u^{-1})) \dots (1+c_{i, l-1} u^{-1})$,
by Lemmas \ref{lanyevenpoly} and \ref{l_sharp4} 
we can  further re-index so that
\eqref{eq_3}, \eqref{eq_4}, \eqref{eq_5}, and \eqref{eq_6} still hold,
$c_{1,j} + c_{1,-j} \le 0$  for $0 \ne j \in \mathcal{I}_l$,
and $c_{1,0} \ge 1$.
Since
$-c_{3,0} \ge c_{1,0} -1$ we now have that
$c_{1,0} - c_{3,0} \ge c_{1,0} + c_{1,0} - 1 \ge 1$.

Now associate to this data the skew-symmetric $n \times l$ tableaux
$A=(a_{i,j})_{i \in \mathcal{I}_n, j \in \mathcal{I}_j}$
indicated by \eqref{lofa}, that is
$a_{i,j} = c_{i,j} -\frac{i}{2}$ for $i \in \{1, 3, \dots, n-1\}$,
$j \in \mathcal{I}_l$, and
$a_{i,j} = - a_{-i,-j}$ for $i  \in \{1-n, 3-n, \dots, -1\}$, $j \in \mathcal{I}_l$,
and now it is clear that $A \in \Col_{n,l}^+$.

Now suppose that we are given $A=(a_{i,j})_{i \in \mathcal{I}_n, j \in \mathcal{I}_j} \in \Col_{n,l}^+$.
Associate to $A$ 
the tuple
$\bar \mu(u) \in (1+u^{-1}\C[[u^{-1}]])^{n/2}$
as indicated
by \eqref{lofa}, where
$\mu_i(u) = (1+c_{i,1-l}u^{-1}) (1+c_{i,3-l} u^{-1}) 
\dots (1+c_{i, l-1} u^{-1}) 
(1+\frac{1}{2} u^{-1})^{-1}$ for all $i \in \{1, 3, \dots, n-1\}$.
So we have that 
$c_{1,j} + c_{1,-j} \le 0$ for $0 \ne j \in \mathcal{I}_l$,
$c_{1,0} + c_{3,0} \le 1$, and
$c_{i,j} \ge c_{i+2,j}$ for $i \in \{1, 3, \dots, n-3\}, j \in \mathcal{I}_l$.
Now it is clear by Lemma \ref{larrow} that
$\mu_1(u) \rightarrow \mu_3(u) \rightarrow \dots \rightarrow \mu_{n-1}(u)$.
Since $c_{1,0} + c_{3,0} \le 1$ and $c_{1,0} - c_{3,0} \ge 0$, we have that
$2c_{1,0} \in \Z$.
If $c_{1,0} \le \frac{1}{2}$ then by Lemma \ref{l_sharp1} 
$\mu_1(-u) \Rightarrow \mu_1(u)$.
If $c_{1,0} \le 0$ then by Lemma \ref{l_sharp2} 
$\frac{2u-1}{2u+1} \mu_1(-u) \Rightarrow \mu_1(u)$.
If $c_{1,0} > \frac{1}{2}$ then by Lemma \ref{l_sharp3}
$\mu_1^\sharp(-u) \Rightarrow \mu_1^\sharp(u)$.
If $c_{1,0} > 0$ then by Lemma \ref{l_sharp4} 
$\frac{2u-1}{2u+1} \mu_1^\sharp(-u) \Rightarrow \mu_1^\sharp(u)$.
So in order to apply Theorem \ref{y_n_plus} to prove that
$L(\bar \mu(u))$ is
finite dimensional, we need to establish that 
$\mu_1^\sharp(u) \rightarrow \mu_3(u)$ in the cases that $c_{1,0} > \frac{1}{2}$
and $c_{1,0} > 0$.
By \eqref{eq_musharp} the only thing we still need to establish is that
$c_{1,k} + c_{3,k} \le 1$ where $-c_{1,k}$ is the $\sharp'$-special element of
$(-c_{1,1-l}, -c_{3-l}, \dots, -c_{l-1} )$.
If $k=0$ then we are done, so assume $k \ne 0$.  By Lemma \ref{l_sharp_max}
we have that $c_{1,k} \le c_{1,0}$.  If $c_{3,k} \le c_{3,0}$, then
$c_{1,k} + c_{3,k} \le c_{1,0} + c_{3,0} \le 1$, hence the lemma is proved.
If $c_{3,k} > c_{3,0}$ then we can re-index the rows of 
 $(c_{i,j})_{i \in \{1, 3, \dots, n-1\}, j \in \mathcal{I}_l}$
by interchanging $c_{i,k}$ with $c_{i,0}$ for all $i \in \mathcal{I}_n$,
$i > 1$.
Now we have that $c_{1,k} + c_{3,k} \le 1$, so
$\mu_1^\sharp(u) \rightarrow \mu_3(u)$,
so $L(A)$ is finite dimensional
by Theorem \ref{y_n_plus}.
\end{proof}

\begin{Lemma} \label{plus1}
Theorem \ref{t_class} holds in the case that $\phi = +, \eps = -$, $n$ is even,
and $l$ is even.
\end{Lemma}
\begin{proof}
Let $e'$ be a nilpotent element in $\mf{g}_{n(l+1)}^+$ so that
the Jordan type of $e'$ is $((l+1)^n)$.
By Corollary \ref{pluswalg} 
every $U(\mf{g},e)$-module is a  $U(\mf{g}_{n(l+1)}^+,e')$-module, so we 
need only determine which finite dimensional irreducible 
 $U(\mf{g}_{n(l+1)}^+,e')$-modules factor through $\zeta$.
By \eqref{ker2}
the finite dimensional
$U(\mf{g}_{n(l+1)}^+,e')$-module  $L(\bar \mu(u))$
where
$\mu_i(u) = (1+c_{i,-l}u^{-1}) (1+c_{i,2-l} u^{-1}) \dots 
(1+c_{i, l} u^{-1}) 
 (1+\frac{1}{2} u^{-1})^{-1}$
factors through 
$\zeta$  
precisely when
$\mu_i(u)$ is a polynomial of degree $l$ or less
for all $i \in \{1, 3, \dots, n-1\}$. 
This implies that for all $i \in \{1, 3, \dots, n-1\}$ 
there exists $k \in \mathcal{I}_{l+1}$ such that $c_{i,k} = 1/2$. 
Now if $A^+$ is the skew symmetric $n \times (l+1)$ tableaux associated to
this data as in \eqref{lofa}, then we can permute entries within rows
so that $A^+$ has middle column
\[
  \frac{n}{2} - 1, \frac{n}{2} -2, \dots, 1, 0,0, -1, -2, \dots, 1-\frac{n}{2
},
\] 
which implies the lemma.
\end{proof}

\begin{Lemma}
Theorem \ref{t_class} holds in the case that $\phi = +, \eps = +$, $n$ is odd,
and $l$ is odd.
\end{Lemma}
\begin{proof}
 By \eqref{ker2} we have that an irreducible highest weight
$Y_n^+$-module  $L(\bar \mu(u))$
factors through
$\kappa_l$ 
if
$\mu_i(u) (1+\frac{1}{2} u^{-1})$ 
 is a polynomial of degree $l$ or less
for all $i \in \{1, 3, \dots, n-1\}$. 
So for all $i \in \{1, 3, \dots, n-1\}$
we can write
$\mu_i(u) = (1+c_{i,1-l}u^{-1}) (1+c_{i,3-l} u^{-1}) 
\dots (1+c_{i, l-1} u^{-1}) 
 (1+\frac{1}{2} u^{-1})^{-1}$ for all $i \in \{1, 3, \dots, n-1\}$.
Additionally, since 
$\mu_0(u) \in 1 + u^{-2} \C[[u^{-2}]]$, 
we must have that $\mu_0(u)$ is a polynomial of degree $l-1$ or less,
and we can re-index so that
\begin{equation} \label{eq_case5}
c_{0,j} = - c_{0,-j} \quad \text{for} \quad  0 \ne j \in \mathcal{I}_l
\quad \text{and} \quad c_{0,0} = \frac{1}{2}.
\end{equation}

If condition (i) holds from Theorem \ref{y_n_odd} then by
Lemma \ref{lpoly}  
we can re-index so that \eqref{eq_case5} holds and 
$c_{i,j} \ge c_{i+2,j}$ 
for $i \in \{0, 2, \dots, n-3\}$, $j \in \mathcal{I}_l$.
In particular, we now have that
$\frac{1}{2} \ge c_{2,0} \ge c_{4,0} \ge \dots \ge c_{n-1,0}$.
Associate to this data the skew-symmetric $n \times l$ tableaux 
$A=(a_{i,j})_{i \in \mathcal{I}_n, j \in \mathcal{I}_j}$
where 
$a_{0,0} = 0$,
$a_{i,j} = c_{i,j} -\frac{i}{2}$ for $i \in \{0, 2, \dots, n-1\}$,
$j \in \mathcal{I}_l$, $(i,j) \ne (0,0)$, and
$a_{i,j} = - a_{-i,-j}$ for $i  \in \{1-n, 3-n, \dots, -1\}$, $j \in \mathcal{I}_l$.
Now it is clear that $A \in \Col_{n,l}^+$.

If condition (ii) holds from Theorem \ref{y_n_odd} then 
$\mu_0(u) \rightarrow (1+\frac{1}{2} u^{-1}) \mu_2(u) 
    \rightarrow (1+\frac{1}{2} u^{-1}) \mu_4(u) \rightarrow 
\dots  \rightarrow (1+\frac{1}{2} u^{-1}) \mu_{n-1}(u)$.
So by  Lemma \ref{lpoly}
we can re-index so that \eqref{eq_case5} holds,
$0 \ge c_{2,0} \ge c_{4,0} \ge \dots \ge c_{n-1,0}$,
and
$c_{i,j} \ge c_{i+2,j}$ 
for $i \in \{0, 2, \dots, n-3\}$, $0 \ne j \in \mathcal{I}_l$.
Associate to this data the skew-symmetric $n \times l$ tableaux 
$A=(a_{i,j})_{i \in \mathcal{I}_n, j \in \mathcal{I}_j}$
where 
$a_{0,0} = 0$,
$a_{i,j} = c_{i,j} -\frac{i}{2}$ for $i \in \{0, 2, \dots, n-1\}$,
$j \in \mathcal{I}_l$, $(i,j) \ne (0,0)$, and
$a_{i,j} = - a_{-i,-j}$ for $i  \in \{1-n, 3-n, \dots, -1\}$, $j \in \mathcal{I}_l$.
Now it is clear that $A \in \Col_{n,l}^+$.

Now suppose we are given 
$A=(a_{i,j})_{i \in \mathcal{I}_n, j \in \mathcal{I}_j} \in \Col_{n,l}^+$.
Associate to $A$ 
the tuple
$\bar \mu(u) \in (1 + u^{-1} \C[[u^{-1}]])^{(n+1)/2}$
as indicated
by \eqref{lofa}, where
$\mu_i(u) = (1+c_{i,1-l}u^{-1}) (1+c_{i,3-l} u^{-1}) 
\dots (1+c_{i, l-1} u^{-1}) 
(1+\frac{1}{2} u^{-1})^{-1}$ for all $i \in \{0, 2, \dots, n-1\}$.
Now if $c_{2,0} \in \frac{1}{2} \Z \setminus \Z$
then by Lemma \ref{larrow}
$\mu_0(u) \rightarrow \mu_2(u) \rightarrow
\dots  \rightarrow \mu_{n-1}(u)$,
and if $c_{2,0} \in \Z$ then by Lemma \ref{larrow}
$\frac{2u}{2u+1} \mu_0(u) \rightarrow \mu_2(u) \rightarrow \mu_4(u) 
  \rightarrow \dots  \rightarrow \mu_{n-1}(u)$,
so
$L(\bar \mu(u))$ is a finite dimensional
$Y_n^-$-module by Theorem \ref{y_n_odd}.
\end{proof}

\begin{Lemma}
Theorem \ref{t_class} holds in the case that $\phi = +, \eps = -$, $n$ is odd,
and $l$ is even.
\end{Lemma}
\begin{proof}
Mimic the proof of Lemma \ref{plus1}.
\end{proof}

\section{BGK highest weight theory for rectangular finite $W$-algebras}
\label{section5}

In this section we show that certain
irreducible highest weight $U(\mf{g},e)$-modules as defined in \cite{BGK}
are isomorphic to the $U(\mf{g},e)$-modules $L(A)$ for
$A \in \Row_{n,l}$, provided one makes the right choices in 
defining the irreducible highest weight $U(\mf{g},e)$-modules from \cite{BGK}.
Throughout this section, unless otherwise indicated, $\mf{g}$ denotes  an
arbitrary reductive Lie algebra over $\C$, and $e \in \mf{g}$ is  
a nilpotent element for which the grading from \eqref{eq_grading} is even,
though all the results we mention
hold in general.  Refer to \cite{BGK} for the general results.

\subsection{Highest weight theory for $U(\mf{g},e)$} \label{s_II_4}
In \cite{BGK}
Brundan, Goodwin, and Klesh\-chev define the notion of a highest weight 
$U(\mf{g},e)$-module.
The key to this is a reductive subalgebra
$\mf{g}_0$ of $\mf{g}$ which contains $e$.
This leads to the ``smaller'' finite $W$-algebra $U(\mf{g}_0,e)$
which plays the role
of a Cartan subalgebra in defining highest weight modules.

To define $\mf{g}_0$, first choose $\mf{t}$, a maximal toral subalgebra of
$\mf{g}$, so that it contains $h$ and so that $\mf{t}^e$ is a maximal toral subalgebra of
$\mf{g}^e \cap \mf{g}(0)$.  For $\alpha \in (\mf{t}^e)^*$ let $\mf{g}_\alpha$ denote the
$\alpha$-weight space of $\mf{g}$.  So
\[
  \mf{g} = \mf{g}_0 \oplus \bigoplus_{\alpha \in \Phi^e} \mf{g}_\alpha,
\]
where $\mf{g}_0$ is the centralizer of $\mf{t}^e$ in $\mf{g}$ and $\Phi^e \subset (\mf{t}^e)^*$ denotes the set of nonzero weights of
$\mf{t}^e$ on $\mf{g}$. Thus we have defined $\mf{g}_0$,
which is now a minimal Levi subalgebra of $\mf{g}$ containing $e$.

Next we choose a Borel subalgebra $\mf{b}$ of $\mf{g}$ containing
$\mf{t}$, and let $\Phi^+$ denote the corresponding set of positive roots.
Let
$\mf{q} = \mf{g}_0 + \mf{b}$, which is a parabolic subalgebra of $\mf{g}$
with Levi factor $\mf{g}_0$.
For each simple root $\alpha \in \Phi^+$, the corresponding root
space of $\mf{g}$ must lie in $\mf{g}_0$ or the nil-radical of $\mf{q}$.
It follows that
$\mf{g}_\alpha \subseteq \mf{q}$ or $\mf{g}_{-\alpha} \subseteq \mf{q}$ for each 
$\alpha \in \Phi^e$.
Define
$\Phi_+^e = \set{\alpha \in \Phi^e }{ \mf{g}_\alpha \subseteq \mf{q}}$.
This defines the {\em dominance order $\ge$ } on $(\mf{t}^e)^*$:
$\lambda \ge \mu$ if $\lambda - \mu \in \Z_{\ge 0} \Phi_+^e$,
and it is now the case that $\Phi^e = -\Phi_+^e \sqcup \Phi_+^e$.

Let $\mf{a}$ be $\mf{g}, \mf{g}^e$, or $\mf{p}$,
and for $\alpha \in (\mf{t}^e)^*$ let $\mf{a}_\alpha$ denote the
$\alpha$-weight space of $\mf{a}$.
Let $\mf{a}_\pm = \bigoplus_{\alpha \in \Phi_\pm^e} \mf{a}_\alpha$,
so $\mf{a} = \mf{a}_- \oplus \mf{a}_0 \oplus \mf{a}_+$ and
$U(\mf{a}) = \bigoplus_{\alpha \in \Z \Phi^e} U(\mf{a})_\alpha$.
In particular, $U(\mf{a})_0$ is a subalgebra.
Let $U(\mf{a})_\sharp$ denote the {\em left} ideal of $U(\mf{a})$ generated by the
roots spaces $\mf{a}_\alpha$ for $\alpha \in \Phi_+^e$.  Similarly
let
$U(\mf{a})_\flat$ denote the {\em right} ideal of $U(\mf{a})$ generated by the
roots spaces $\mf{a}_\alpha$ for $\alpha \in \Phi_-^e$.
Let $U(\mf{a})_{0,\sharp} = U(\mf{a})_0 \cap U(\mf{a})_\sharp$, and
 $U(\mf{a})_{\flat,0} = U(\mf{a})_0 \cap U(\mf{a})_\flat$.  Now the PBW theorem
implies that $U(\mf{a})_{0,\sharp} = U(\mf{a})_{\flat,0}$, hence
$U(\mf{a})_{0,\sharp}$ is a two-sided ideal of $U(\mf{a})_0$.  Moreover,
$\mf{a}_0$ is a subalgebra of $\mf{a}$, and we actually have that
$U(\mf{a})_0 = U(\mf{a}_0) \oplus U(\mf{a})_{0,\sharp}$.
Let
\[
  \pi : U(\mf{a})_0 \twoheadrightarrow U(\mf{a}_0)
\]
be the algebra homomorphism defined by projection along this decomposition.

It is easy to see that
\begin{equation} \label{eq_te_embed}
\mf{t}^e \subseteq U(\mf{g},e) 
\end{equation}
since
$([m, t],e) = 0$ for all $m \in \mf{m}, t \in \mf{t}^e$.
Recall that  the good filtration on $U(\mf{g},e)$ is
defined in $\S$\ref{section2.1}, and that $\gr U(\mf{g},e) = U(\mf{g}^e)$.
The following
theorem is due to Premet in \cite{P1}:
\begin{Theorem} \label{t_premet}
  There exists a $\mf{t}^e$-equivariant injection
$\Theta: \mf{g}^e \hookrightarrow U(\mf{g},e)$ such that
$\gr \Theta : \mf{g}^e \hookrightarrow U(\mf{g}^e)$ is the natural embedding.
\end{Theorem}
It should be noted that $\Theta$ is not a Lie algebra homomorphism.

Let $h_1, \dots, h_l$ be a basis of $\mf{g}_0^e$.
Let $f_1, \dots, f_m$, and $e_1, \dots, e_m$ be $\mf{t}^e$-weight bases of
$\mf{g}^e_-$ and $\mf{g}^e_+$ respectively, such that
$f_i$ is of weight $-\gamma_i$, and $e_i$ is of weight $\gamma_i$ for
$\gamma_1, \dots, \gamma_m \in \Phi^e_+$.
For $i = 1, \dots, m$, $j = 1, \dots, l$, let
$F_i = \Theta(f_i), E_i = \Theta(e_i)$, and $H_j = \Theta(h_j)$.
For $\bf{a} \in \Z_{\ge 0}^m$, let $F^{\bf{a}} = F_1^{a_1} \dots F_m^{a_m}$.
For $\bf{b} \in \Z^l_{\ge 0}, \bf{c} \in \Z^m_{\ge 0}$ define $H^{\bf{b}}, E^{\bf{c}}$
similarly.
Theorem \ref{t_premet} implies that the
following is a PBW basis of $U(\mf{g},e)$:
\[
  \set{F^{\bf{a}} H^{\bf{b}} E^{\bf{c}}}{ \bf{a}, \bf{c} \in \Z^m_{\ge 0}, 
  \bf{b} \in \Z^l_{\ge 0} }.
\]

Let $U(\mf{g},e)_\sharp$ be the {\em left} ideal of $U(\mf{g},e)$ generated by
$\{E_1, \dots, E_m\}$.
Let $U(\mf{g},e)_\flat$ be the {\em right} ideal of $U(\mf{g},e)$ generated by
$\{F_1, \dots, F_m\}$.
Let $U(\mf{g},e)_{0, \sharp} = U(\mf{g},e)_\sharp \cap U(\mf{g},e)_0$.
Let $U(\mf{g},e)_{\flat,0} = U(\mf{g},e)_\flat \cap U(\mf{g},e)_0$.
Now from the above PBW basis it is clear that
$U(\mf{g},e)_{0, \sharp} = U(\mf{g},e)_{\flat, 0}$, and so
$U(\mf{g},e)_{0, \sharp}$ is a {\em two-sided} ideal of $U(\mf{g},e)_0$.

Let $b_1, \dots, b_r$ be a homogeneous basis for $\mf{m}$ such that
$b_i$ is of degree $-d_i$ and $\mf{t}$-weight $\beta_i \in \mf{t}^*$, and let
\begin{equation} \label{eq_gamma}
  \gamma = \sum_{\substack{ 1 \le i \le r  \\ \beta_i|_{\mf{t}^e} \in \Phi^e_-}} \beta_i.
\end{equation}
By \cite[Lemma 4.1]{BGK}, $\gamma$ extends uniquely to a
character of $\mf{p}_0$.
Let $S_{-\gamma} : U(\mf{p}_0) \to U(\mf{p}_0)$ be defined by $S_{-\gamma}(x) = x -\gamma(x)$
for $x \in \mf{p}_0$,
so $S_{-\gamma}$ an algebra isomorphism.

\begin{Theorem} 
[{\cite[Theorem 4.3]{BGK}}]
\label{t_bgk_even}
The restriction of
$S_{-\gamma} \circ \pi : U(\mf{p})_0 \twoheadrightarrow U(\mf{p}_0)$
to $U(\mf{g},e)_0$
defines a surjective algebra homomorphism
\[
  \pi_{-\gamma} : U(\mf{g},e)_0 \twoheadrightarrow U(\mf{g}_0,e)
\]
with $\ker \pi_{-\gamma} = U(\mf{g},e)_{0, \sharp}$.
\end{Theorem}


For a $U(\mf{g},e)$-module $V$ and $\lambda \in (\mf{t}^e)^*$
let
\begin{equation} \label{eq_vlambda}
  V_\lambda = \set{v \in V }
    { (t+\gamma(t))v = \lambda(t) \text{ for all } t \in \mf{t}^e},
\end{equation}
recalling that $\mf{t}^e$ is naturally a subalgebra of $U(\mf{g},e)$ by \eqref{eq_te_embed}.
Now it is the case that
$U(\mf{g},e)_\alpha V_\lambda  \subseteq V_{\lambda + \alpha}$, so
$V_\lambda$ is preserved by $U(\mf{g},e)_0$.
We say that $V_\lambda$ is a {\em maximal weight space of
$V$} if $U(\mf{g},e)_\sharp V_\lambda = 0$.
Assuming this is the case, the action of $U(\mf{g},e)_0$ factors through the homomorphism   $\pi_{-\gamma}$ from
Theorem \ref{t_bgk_even}, thus $V_\lambda$ is also a $U(\mf{g}_0,e)$-module.
Since $\mf{t}^e$ can naturally be considered a subalgebra of $U(\mf{g}_0,e)$
by \eqref{eq_te_embed} again,
restricting the action of $U(\mf{g}_0,e)$ on $V_\lambda$  to
$\mf{t}^e$ gives a new action of $\mf{t}^e$ on $V_\lambda$ satisfying
\[
   t.v = \lambda(t)v \quad \text{for all} \quad t \in \mf{t}^e
\]
(which is why the shift by $\gamma$ is included in the definition of the
$\lambda$-weight space of a $U(\mf{g},e)$-module from \eqref{eq_vlambda}).

A $U(\mf{g},e)$-module is a {\em highest weight module} if it is generated by
a maximal weight space $V_\lambda$ such that $V_\lambda$ is finite dimensional
and irreducible as a $U(\mf{g}_0,e)$-module.
Let
\[
  \set{V_\Lambda }{ \Lambda \in \mathcal{L}}
\]
be a complete set of isomorphism classes of finite dimensional irreducible
$U(\mf{g}_0,e)$-modules for some indexing set $\mathcal{L}$.
Since $U(\mf{g},e)_\sharp$ is invariant under left multiplication by $U(\mf{g},e)$
and right multiplication by $U(\mf{g},e)_0$, we have that
$U(\mf{g},e) / U(\mf{g},e)_\sharp$ is a $(U(\mf{g},e), U(\mf{g},e)_0)$-bimodule.  Moreover
the right action of $U(\mf{g},e)_0$ factors through the homomorphism
$\pi_{-\gamma}$
from Theorem \ref{t_bgk_even}.
 Thus we have that
$U(\mf{g},e) / U(\mf{g},e)_\sharp$ is a $(U(\mf{g},e), U(\mf{g}_0,e))$-bimodule.
For $\Lambda \in \mathcal{L}$, define $M(\Lambda, \mf{q})$, the {\em Verma module of type $\Lambda$} via
\[
  M(\Lambda, \mf{q}) = U(\mf{g},e)/U(\mf{g},e)_\sharp \otimes_{U(\mf{g}_0,e)} V_\Lambda.
\]

By \cite[Theorem 4.5]{BGK} $M(\Lambda, \mf{q})$ has a unique maximal proper submodule
$R(\Lambda, \mf{q})$.
Let $L(\Lambda, \mf{q}) = M(\Lambda,\mf{q}) / R(\Lambda, \mf{q})$.
Now also by
\cite[Theorem 4.5]{BGK}
we have that
$\set{L(\Lambda, \mf{q}) }{ \Lambda
 \in \mathcal{L}}$
is a complete set of isomorphism classes of irreducible highest weight modules for $U(\mf{g},e)$.
Let
\[
\mathcal{L}^+ = \set{\Lambda \in \mathcal{L} }{ 
\dim L(\Lambda,\mf{q}) < \infty}.
\]
By \cite[Corollary 4.6]{BGK},
$\set{ L(\Lambda,\mf{q}) }{ \Lambda \in \mathcal{L}^+ }$
 is a complete set of isomorphism classes of
finite dimensional irreducible $U(\mf{g},e)$-modules.

Unfortunately, an explicit set $\mathcal{L}$ parameterizing the finite dimensional
irreducible $U(\mf{g}_0,e)$-modules is unknown in general.  In the next subsection, we
focus on a special case in which such a parameterization is available.

\subsection{The case that $e$ is regular in $\mf{g}_0$} \label{s_II_5}
We assume in this section that $e$ is a regular nilpotent element of $\mf{g}_0$.
 In this case,
Kostant showed in \cite[$\S$2]{Ko} that
$U(\mf{g}_0,e) \cong Z(\mf{g}_0)$.
In turn, by the Harish-Chandra Isomorphism,
$Z(\mf{g}_0) \cong S(\mf{t})^{W_0}$ where $W_0$ is the Weyl group associated
to $\mf{g}_0$.
We state this more precisely in the following lemma.
Let
\[
  \eta = \frac{1}{2} \sum_{\substack{\alpha \in \Phi \\ \alpha|_{\mf{t}^e} \in \Phi^e_+}} \alpha + 
   \frac{1}{2} \sum_{\substack{1 \le i \le r \\ \beta_i|_{\mf{t}^e} =0}} \beta_i,
\]
where the $\beta_i$ are defined as in \eqref{eq_gamma}.
The following lemma is essentially \cite[Lemma 5.1]{BGK}:
\begin{Lemma} \label{l_hc}
 Let $\xi:U(\mf{p}_0)\to S(\mf{t})$ be the homomorphism induced by the
natural projection $\mf{p}_0 \twoheadrightarrow \mf{t}$.
Let $S_{-\eta}: S(\mf{t}) \rightarrow S(\mf{t}), x \mapsto x - \eta(x)$
for $x \in \mf{t}$.
Then the map $\xi_{-\eta} := S_{-\eta} \circ \xi$ defines an
algebra isomorphism
$U(\mf{g}_0,e) \stackrel{\sim}{\rightarrow} S(\mf{t})^{W_0}$.
\end{Lemma}

Since $S(\mf{t})^{W_0}$ is a free polynomial algebra,
we have by the isomorphism from Lemma \ref{l_hc} that $\mathcal{L} = \mf{t}^* / W_0 = \operatorname{m-Spec}(S(\mf{t})^{W_0})$.
In this case 
we can describe
the subset $\mathcal{L}^+$ of $\mathcal{L}$,
corresponding to the finite dimensional irreducible $U(\mf{g}_0, e)$-modules
$\Lambda$ for which $L(\Lambda, \mf{q})$ is finite dimensional,
in combinatorial terms.
Recall that we have
fixed a Borel subalgebra $\mf{b}$ of $\mf{g}$ containing $\mf{t}$, and $\Phi^+$ is the
corresponding set of positive roots. Let
$\Phi_0^+ = \{\alpha \in \Phi^+\:|\:\mf{g}_\alpha \subseteq \mf{g}_0\}$
denote the resulting system of
positive roots for the Levi subalgebra $\mf{g}_0$ of $\mf{g}$.
For $\lambda \in \mf{t}^*$ let
$L(\lambda)$ denote the irreducible $U(\mf{g})$-module of highest weight
$\lambda - \rho$, where $\rho = \frac{1}{2} \sum_{\alpha \in \Phi^+} \alpha$.

\begin{Theorem}[{\cite[Conjecture 5.2]{BGK}, proved by Losev in \cite{Lo3}}]
\label{c_bgk}
  For $\Lambda \in \mathcal{L}$ pick $\lambda \in \Lambda$ such that
  $(\lambda,  \alpha^\vee) \notin \Z_{>0}$ for all $\alpha \in \Phi_0^+$.
  Then $L(\Lambda, \mf{q})$ is finite dimensional if and only if
$\mathcal{VA}(\Ann_{U(\mf{g})} L(\lambda)) = \overline {G.e}$.
\end{Theorem}

\begin{Remark}
In the course of this work we also understood how to apply the results of
this chapter and the algorithms 
for calculating $\mathcal{VA}(\Ann_{U(\mf{g})} L(\lambda))$ from \cite{BV}
to independently verify that Theorem \ref{c_bgk} holds in these cases. 
In fact it is possible to use Losev's proof of Conjecture \ref{c_bgk} and
\eqref{ker2}
 to recover Molev's classification of the finite dimensional irreducible
representations of $Y$ from the classification of the finite dimensional irreducible
representations of $U(\mf{g},e)$ obtained via BGK highest weight theory.
\end{Remark}

\subsection{BGK highest weight theory for rectangular finite $W$-algebras}
In this subsection we show how to identify the irreducible 
$U(\mf{g},e)$-module $L(A)$ for $A \in \Row_{n,l}$ with 
a BGK highest weight module.
For a rectangular finite $W$-algebra $U(\mf{g},e)$ we have that
$e$ is regular in $\mf{g}_0$, 
so Theorem \ref{c_bgk} applies to these finite $W$-algebras.

First we need to fix choices of $\mf{t}$, a Cartan subalgebra of $\mf{g}$, 
and $\mf{b}$, a Borel subalgebra of $\mf{g}$ as in $\S$\ref{s_II_5}.
We let
$\mf{t}$ be the span of diagonal matrices in $\mf{g}$.
We choose our Borel subalgebra $\mf{b}$ by specifying a 
system of positive roots.
For $a \in \mathcal{I}_{nl}$ 
let $\eps_a \in \mf{t}^*$ be the restriction to $\mf{t}$ of the
diagonal coordinate function of $\mf{g}_{nl}$ given by
$\eps_a(e_{b,b}) = \delta_{a,b}$.
If $\eps = -$ (so $n l$ is even) our positive root system is
\begin{align*}
          \Phi^+ &= \set{\eps_a - \eps_b }{ a,b \in \mathcal{I}_{nl}, 
                   \row(a) < \row(b)}    \\
          & \quad \quad \cup 
          \set{\eps_a - \eps_b }{ a,b \in \mathcal{I}_{nl}, 
                   \row(a) = \row(b), \col(a) < \col(b)}.  
\end{align*}
If $\eps = +$ then 
\begin{align*}
          \Phi^+ &= \set{\eps_a - \eps_b }{ a,b \in \mathcal{I}_{nl}, 
                   \row(a) < \row(b), a \ne -b}    \\
          & \quad \quad \cup 
          \set{\eps_a - \eps_b }{ a,b \in \mathcal{I}_{nl}, 
                   \row(a) = \row(b), \col(a) < \col(b), a \ne -b}.
\end{align*}
Let $\mf{b}$ be the Borel subalgebra of $\mf{g}$ 
corresponding to this
choice of positive roots.

Next we give an explicit basis for $\mf{t}^e$, 
the centralizer of $e$ in $\mf{t}$.  
In \cite[Lemma 3.2]{Br} a basis for $\mf{g}^e$ is given in terms
of certain elements $\{f_{i,j;r}\}$, where $f_{i,j;r}$ is nilpotent
unless $r=0$.
So by \cite[Lemma 3.2]{Br}
a basis for $\mf{t}^e$ is given by 
\[
     \set{f_{i,i;0} }{ i \in \mathcal{I}_n \cap \Z_{<0}}.
\]
More explicitly, for $i \in \mathcal{I}_n$ we have that
\[
  f_{i,i;0} = \sum_{\substack{ a \in \mathcal{I}_{nl}\\ \row(a) = i}} f_{a,a}.
\]

Next we give basis for $(\mf{t}^e)^*$.  
Let $\delta_i \in (\mf{t}^e)^*$
be defined via $\delta_i(f_{j,j,;0}) = \delta_{i,j}$ 
for $i, j \in \mathcal{I}_n$, $i, j < 0$,
and for $i > 0$ let $\delta_i = -\delta_{-i}$.

Now $\mf{g}_0$ is the span of 
$\set{f_{a,b} }{ a, b \in \mathcal{I}_{nl}, \row(a) = \row(b) }$,
so
\[
  \mf{g}_0 \cong 
  \begin{cases}
      \mf{g}_l ^ {\oplus n/2} & \text{ if $n$ is even;} \\
      \mf{g}_l^\eps \oplus \mf{g}_l^{\oplus (n-1)/2} & \text{if $n$ is odd.}
  \end{cases}
\]
We also have that the parabolic $\mf{q} = \mf{b} + \mf{g}_0$ 
is the span of
\[  
   \set{ f_{a,b} }{ a,b \in \mathcal{I}_{nl}, \row(a) \le \row(b) }.
\]

Note that for $a,b \in \mathcal{I}_{nl}$, we have that 
$f_{a,b} \in \mf{g}_{\delta_{\row(b)} - \delta_{\row(a)}}$.
Thus 
\[
  \Phi^e_+ = \set{ \delta_{i} - \delta_j }{ i, j \in \mathcal{I}_n, i < j }.
\]

Recall for $i,j \in \mathcal{I}_n$ that there is a map 
$s_{i,j} : T(\mf{g}_l) \to U(\mf{g})$ defined in \eqref{eq_sij}.
This definition makes it clear that
for any $v \in T(\mf{g}_l)$,
$s_{i,j}(v) \in U(\mf{g})_{\delta_j - \delta_i}$.
Thus we can explicitly state a choice for the  $\mf{t}^e$-equivariant map 
$\Theta : \mf{g}^e \to U(\mf{g},e)$ from \eqref{eq_te_embed}:
For $i,j \in \mathcal{I}_n$, $r \ge 0$ we set
$\Theta(f_{i,j;r}) = s_{i,j}(\omega_{r+1})$.
Thus $\set{s_{i,j}(\omega_{r+1}) }{ r \ge 0, i,j \in \mathcal{I}_n, i < j}$ generates the left $U(\mf{g},e)$ ideal $U(\mf{g},e)_\sharp$.

Recall the homomorphisms
$\pi_{-\gamma} : U(\mf{g},e)_0 \twoheadrightarrow U(\mf{g}_0,e)$ 
from Theorem \ref{t_bgk_even}
and
$\xi_{-\eta}: U(\mf{g}_0,e) \widetilde{\to} S(\mf{t})^{W_0}$ from Lemma \ref{l_hc}.
These maps make every $S(\mf{t})^{W_0}$-module a $U(\mf{g},e)_0$-module.
We need to calculate the action of 
$s_{i,i}(\omega_{r+1})$ on a $S(\mf{t})^{W_0}$-module,
so we need to calculate $\xi_{-\eta} \circ \pi_{-\gamma} (s_{i,i}(\omega_r))$.
We do this with a series of lemmas.

For $i \in \Z$ let $\tilde \imath = \widehat{- \imath}$.
This lemma is a special case of \cite[Lemma 4.1]{Br}:
\begin{Lemma} \label{l_comm}
For $i, j \in \mathcal{I}_n, p,q \in \mathcal{I}_l$
\begin{align*}
  [&s_{i,j}(e_{p,q}), s_{h,k}(e_{v,w})] \\
   &=  \delta_{h,j} \delta_{q,v} s_{i,k}(e_{p,w}) -  \delta_{i,k} \delta_{p,w} s_{h,j}(e_{v,q}) \\
   &\quad + \iota(-\delta_{h,-i} \delta_{v,-p}s_{-j,k}(e_{-q,w}) + \delta_{-j,k} \delta_{w,-q}  s_{h,-i}(e_{v,-p}) ),
\end{align*}
where
\begin{equation} \label{iota}
  \iota = \begin{cases}
     \phi^{\hat \imath \hat p + \tilde \imath \tilde{p} +
            \hat \jmath \hat q + \tilde \jmath \tilde{q}}
     \eps^{\hat p + \hat q} & \text{if $p,q \ne 0$;} \\
     \phi^{\hat \jmath \hat q + \tilde \jmath \tilde{q}}
     \eps^{\hat \imath + \hat q} & \text{if $p = 0$, $q \ne 0$;} \\
     \phi^{\hat \imath \hat p + \tilde \imath \tilde{p}}
      \eps^{\hat p + \hat \jmath} & \text{if $p \ne 0$, $q = 0$;} \\
      \eps^{\hat \imath + \hat \jmath} & \text{if $p,q = 0$.} 
   \end{cases}
\end{equation}
\end{Lemma}

Note that $s_{i,i}(\omega_r)$ is a linear combination of monomials of the form
\begin{equation} \label{eq_sijmon}
   s_{i,i_1}(e_{p_1,q_1}) s_{i_1, i_2}(e_{p_2,q_2}) 
    \dots s_{i_{m-1}, i}(e_{p_m, q_m}),
\end{equation}
where $i_j \in \mathcal{I}_n$ for $j=1,\dots, m-1$, $p_i \le q_i$ for 
$i=1,\dots, m$, and $q_i < p_{i+1}$ for $i=1,\dots, m-1$.  
So to calculate $\xi_{-\eta} \circ \pi_{-\gamma} (s_{i,i}(\omega_r))$
we first prove a lemma about applying 
$\pi: U(\mf{p})_0 \to U(\mf{p}_0)$ to such monomials.

\begin{Lemma} \label{l_pi}
Let 
\[
   v = s_{i,i_1}(e_{p_1,q_1}) s_{i_1, i_2}(e_{p_2,q_2}) 
    \dots s_{i_{m-1}, i}(e_{p_m, q_m})
\]
be as in \eqref{eq_sijmon}.
If $i \ge 0$ then $\pi(v) = 0$ unless $i_1 = i_2 = \dots = i_{m-1} = i$.
\end{Lemma}
\begin{proof}
For uniformity, let $i_0, i_m = i$.  
The key fact used repeatedly in this proof is that
if 
$w = s_{j_1, j_2}(e_{r_1, r_2}) 
   \dots s_{j_{k}, j_{k+1}}(e_{r_k, r_{k+1}}) \in U(\mf{p})_0$ 
satisfies $j_1 > j_2$ or $j_k < j_{k+1}$ then
$w \in U(\mf{p})_{0,\sharp} = U(\mf{p})_{\flat,0}$, so
$\pi(w) = 0$.

By Lemma \ref{l_comm} we see that each term $s_{i_{j-1}, i_{j}}(e_{p_j,q_j})$ 
of $v$
commutes
with all terms 
$s_{i_{k-1}, i_{k}}(e_{p_k,q_k})$ in $v$ unless $p_k=q_j$, $q_k=p_j$, 
$p_k=-p_j$, or $q_k=-q_j$.

Suppose that there exists $j$ such that 
$i_{j-1} < i_j$ and $p_j,q_j > 0$.  Then 
$s_{i_{j-1}, i_j}(e_{p_j,q_j})$ commutes with every term to its right, so 
$\pi(v)=0$.
Next suppose that there exists a $j$ such that
$i_{j-1} > i_j$ and $p_j,q_j < 0$.  Then 
$s_{i_{j-1}, i_j}(e_{p_j,q_j})$ commutes with every term to its left
in $v$, so $\pi(v) = 0$.  
So $\pi(v) = 0$ unless $v$ satisfies
$i_{j-1} \le i_j$ if $q_j \le 0$ and $i_{j-1}  \ge i_j$ if $p_j \ge 0$, 
so {\em for the rest of this proof we assume that this is the case}.

Now suppose that there exists a $j$ such that $p_j < 0$, $q_j > 0$ and $i_{j-1} < i_j$. 
Then for all $k$ we must have that $i_k \ge 0$.  Note that
$s_{i_{j-1}, i_j} (e_{p_j,q_j})$ 
must commute with every term to its right unless there exists
$k>j$ such that $p_k = -p_j$.  In this case, 
$[s_{i_{j-1}, i_j} (e_{p_j,q_j}), s_{i_{k-1}, i_k} (e_{p_k,q_k})]$ 
is a multiple of 
$s_{-i_j,i_k}(e_{-q_j, q_k})$, which commutes with every term to the right of 
$s_{i_{k-1}, i_k} (e_{p_k,q_k})$.  
Furthermore since $i_j > 0$ and $i_k \ge 0$ we have that $-i_j < i_k$.  Thus
$\pi(v)=0$.  

Next suppose that there exists a $j$ such that 
$p_j < 0$, $q_j > 0$ and $i_{j-1} > i_j$.
Then for all $k$ we must have that $i_k \ge 0$.  Note that 
$s_{i_{j-1}, i_j} (e_{p_j,q_j})$ 
must commute with every term to its left unless there exists
$k<j$ such that $q_k = -q_j$.  In this case,
$[s_{i_{k-1}, i_k} (e_{p_k,q_k}), s_{i_{j-1}, i_j} (e_{p_j,q_j})]$ 
is a multiple of 
$s_{i_{j-1}, -i_{k-1}} (e_{p_j, -p_k})$, 
which commutes with every term to the left of
$s_{i_{k-1}, i_k} (e_{p_k,q_k})$, and it also satisfies $i_{j-1} > -i_{k-1}$.  
Thus $\pi(v) = 0$.

So for the rest of the proof we will assume that 
if there exists a $j$ such that $p_j <0$ and
$q_j > 0$ then $i_{j-1} = i_j$.  

Let $j$ be such that $i_{j-1} < i_j$ and $i_j$ is 
maximal in $\{i_1, \dots, i_m\}$.  Now it must be the case that $q_j < 0$.
Since $i_k \ge 0$ for all $k$,  by Lemma \ref{l_comm}
$s_{i_{j-1}, i_j}(e_{p_j, q_j})$ 
must commute with every term to its right unless $i=0$.
So if $i \ne 0$, then $\pi(v)=0$, so assume that $i=0$.
Even in the case that $i=0$, since $i_j > 0$ we still have that 
$s_{i_{j-1}, i_j}(e_{p_j, q_j})$ 
commutes with every term to its right unless there 
exists a $k > j$ such that $p_k = - p_j$.  In this case,
$[s_{i_{j-1}, i_j} (e_{p_j,q_j}), s_{i_{k-1}, i_k} (e_{p_k,q_k})]$ 
is a multiple of 
$s_{-i_j,i_k}(e_{-q_j, q_k})$
which commutes with every term to the right of 
$s_{i_{k-1}, i_k} (e_{p_k,q_k})$.  Furthermore, $-i_j < i_k$ since $i_j > 0$,
so
$\pi(v)=0$.

Thus we have proven that
$\pi(v) \ne 0 $ if and only if $i_0 = i_1 = i_2 = \dots = i_m = i$.
\end{proof}

Observe
that if
\[
   v = s_{i,i_1}(e_{p_1,q_1}) s_{i_1, i_2}(e_{p_2,q_2}) 
    \dots s_{i_{m-1}, i}(e_{p_m, q_m}) \in U(\mf{p}_0)
\]
then
\begin{equation} \label{l_xi}
\xi(v)=0  \text{ unless } p_j = q_j \text{ for } j=1, \dots, m.
\end{equation}

\begin{Lemma} \label{l_shift}
Let $i \in \mathcal{I}_n \cap \Z_{\ge 0}, p \in \mathcal{I}_l$.
Then
\[
  S_{-\eta}(S_{-\gamma}(s_{i,i}(e_{p,p} + \rho_p))) = 
  \begin{cases}
     s_{i,i}(e_{p,p})+ \frac{i}{2} & \text{ if $p \ne 0$;} \\
        s_{i,i}(e_{p,p}) + \frac{i}{2} - \frac{\eps}{2} & \text{ if $p=0$, $i \ne 0$;} \\
        s_{i,i}(e_{p,p}) & \text{ if $p,i=0$;}
   \end{cases}
\]
\end{Lemma}
\begin{proof}
Recall that the weight $\gamma$ is defined by 
choosing a weight basis $\{b_1, \dots, b_r\}$ for $\mf{m}$,
where each $b_i$ is of weight $\beta_i \in \mf{t}^*$.
A natural basis to choose is
\[
  \set{ f_{a,b} }{ a,b \in \mathcal{I}_{nl}, 
    (a,b) \text{ is admissible, }  \col(a) > \col(b) },
\]
where $(a,b)$ is admissible if
$a + b <0$ if $\eps = 1$ and $a+b \le 0$ if $\eps = -1$.
Note that $f_{a,b}$ is of weight $\eps_a - \eps_b$.
Recall that $\gamma$ is now defined by
\[
  \gamma = \sum_
  {\substack{ 1 \le i \le r \\ \beta_i|_{\mf{t}^e} \in \Phi_-^e}} 
     \beta_i
.
\]
Thus in these cases
\begin{equation} \label{gamma}
\gamma = \sum_{\substack{a,b \in \mathcal{I}_{nl} \\
                     \col(a) > \col(b) \\ 
                      \row(a) > \row(b) \\
                      (a,b) \text{ is admissible}}} \eps_a - \eps_b.
\end{equation}
So we have  for $a \in \mathcal{I}_{nl}$ that
\[
  \gamma(f_{a,a}) = 
   \begin{cases}
           \RA(a)\CL(a) - \RB(a) \CR(a)
         - \eps
    & \text{ if $a$ is in the lower right quadrant;}  \\
         \RA(a)\CL(a) - \RB(a) \CR(a)
         + \eps
     & \text{ if $a$ is in the upper left quadrant;} \\
    \RA(a)\CL(a) - \RB(a) \CR(a)
    & \text{ otherwise,} \\
  \end{cases}  
\]
where $\RA(a)$ denotes the number of rows occurring strictly 
above the number $a$ in
the $n \times l$ array used to define $U(\mf{g},e)$ in
$\S$\ref{section2.1}, $\RB(a)$ denotes the number of rows
strictly
below $a$, $\CL(a)$ denotes the number of columns strictly left of $a$, and
$\CR(a)$ denotes the number of columns strictly to the right of $a$.  
Also, to be clear, by lower right quadrant we mean the boxes in 
the array from $\S$\ref{section2.1} which are in positions
$(i,j)$ where $\row(i), \col(j) > 0$, and similarly for upper left quadrant.
In calculations below we use the following simplification:
\[ 
  \RA(a)\CL(a) - \RB(a) \CR(a) = \frac{1}{2}((n-1)\col(a) + (l-1)\row(a)).
\]

Now we turn our attention to the shift $S_\eta$.
Recall that
$\eta = \frac{1}{2} \eta_1 + \frac{1}{2} \eta_2$
where
\[
\eta_1 = \sum_{\substack{\alpha \in \Phi \\ 
      \alpha|_{\mf{t}^e} \in \Phi^e_+}} \alpha 
\quad \text{ and } \quad
\eta_2  =  \sum_{\substack{1 \le i \le r \\ \beta_i|_{\mf{t}^e} =0}} \beta_i.
\]
Now we calculate for $a \in \mathcal{I}_{nl}$ that
  \[
  \eta_1 =  l(\RB(a)-\RA(a)) + \delta_{i,0} \sgn(i) = - l \row(a) + \delta_{i,0} \sgn(i). 
  \]
where $i = \row(a)$.
Also we calculate using the fact that $\CL(a) - \CR(a) = \col(a)$ to get that
\[
  \eta_2(f_{a,a}) =
  \begin{cases}
    \col(a) & \text{ if $\row(a) \ne 0$;} \\
    \col(a) - \eps & \text{ if $\row(a) = 0$, and
       $\col(a) > 0$}; \\
    \col(a) + \eps & \text{ if $\row(a) = 0$, and
       $\col(a) < 0$.}
  \end{cases}
\]

Now we are ready to calculate $\gamma(f_{a,a}) + \eta(f_{a,a})$.  
If $a$ is in the
lower right quadrant
then we calculate using the above results to get that
\begin{align*}
  \gamma(f_{a,a}) + \eta(f_{a,a}) &=
   \frac{1}{2}((n-1)\col(a) + (l-1)\row(a))-\eps \\
    &\quad
    + \frac{1}{2}(-l \row(a) +\eps)
    +\frac{1}{2} \col(a) \\
   &= \frac{1}{2}(-\eps+n \col(a) - \row(a)).
\end{align*}

Similar calculations show that
if $a$ is in the bottom half of the middle column then
\begin{align*}
  \gamma(f_{a,a}) + \eta(f_{a,a}) &=
    \frac{1}{2}(\eps+n \col(a) - \row(a)),
\end{align*}
if $a$ is in the right half of the middle row then
\begin{align*}
  \gamma(f_{a,a}) + \eta(f_{a,a}) &=
     \frac{1}{2}(-\eps+n \col(a) - \row(a)),
\end{align*}
and if $a$ is in the upper right
then
\begin{align*}
  \gamma(f_{a,a}) + \eta(f_{a,a}) &=
    \frac{1}{2}(-\eps+n \col(a) - \row(a)).
\end{align*}

So in all cases 
we have that
\[
  S_{-\eta}(S_{-\gamma}(s_{i,i}(e_{p,p} + \rho_p))) = 
  \begin{cases}
      s_{i,i}(e_{p,p}) + \frac{i}{2} & \text{if $p \ne 0$;} \\
        s_{i,i}(e_{p,p}) + \frac{i}{2} - \frac{\eps}{2} & \text{ if $p=0$, $i \ne 0$;} \\
        s_{i,i}(e_{p,p})  & \text{ if $p,i=0$.}
   \end{cases}
\]
\end{proof}

Let $E_i^{(r)}$ denote
 the $r$th elementary symmetric function
in 
\[
   \set{ f_{a,a} + \frac{\row(a)}{2} }{ a \in \mathcal{I}_{nl}, \row(a) = i, \col(a) \in \mathcal{I}_l}.
\]
Recall the definition of $\omega_r$ from \eqref{omegadef}
\begin{Lemma} \label{l_elem}
Let $i \in \mathcal{I}_n$ 
If $i > 0$, and $l$ is even then
\[
    S_{-\eta} \circ \pi_{-\gamma}(s_{i,i}(\omega_r)) = E_i^{(r)}.
\]
If $i\ge 0$ and $l$ is odd then
\[
S_{-\eta}  \circ \pi_{-\gamma}(s_{i,i}(\omega_r)) = \sum_{i=0}^{r-1}
      (-2 \eps)^{i} E_i^{(r-i)}.
\]
\end{Lemma}
\begin{proof}
If $l$ is even then by
Lemma \ref{l_pi}, \eqref{l_xi}, and Lemma \ref{l_shift}
we have that 
\begin{equation} \label{bgkleven}
    S_{-\eta} \circ \pi_{-\gamma}(s_{i,i}(\omega(u)) =
    s_{i,i}(u+e_{1-l, 1-l} + i/2 ) \dots s_{i,i}(u+e_{l-1, l-1} + i/2)
\end{equation}
so the lemma holds in this case.

Now we consider the $l$ odd case.  Let
 \begin{align*}
P_i(u) &= 
  s_{i,i}(u+e_{1-l, 1-l} + i/2 )   \dots s_{i,i}(u+e_{-2,-2} + i/2) \\
     &\qquad \times  s_{i,i}(u+e_{0,0} + i/2 - \eps / 2)
      s_{i,i}(u+e_{2,2}+i/2) +\dots s_{i,i}(u+e_{l-1, l-1} + i/2) ,
\end{align*}
and
\begin{align*}
Q_i(u) &= 
  s_{i,i}(u+e_{1-l, 1-l} + i/2 )   \dots s_{i,i}(u+e_{-2,-2} + i/2) \\
     &\qquad \times  s_{i,i}(e_{0,0} + i/2 - \eps / 2) s_{i,i}(u+e_{2,2}+i/2) +\dots s_{i,i}(u+e_{l-1, l-1} + i/2) .
\end{align*}
So
\begin{equation*}
 S_{-\eta} \circ \pi_{-\gamma}(s_{i,i}(\omega(u)) =P_i(u)  +
      \sum_{r=1}^\infty (-2 \eps u)^{-r} Q_i(u).
\end{equation*}
Observe that $P_i(u) = P'_i(u) - \frac{\eps}{2}P''_i(u)$ and 
$Q_i(u) = Q'_i(u) - \frac{\eps}{2} P''_i(u)$ where
\begin{align} \label{pprime}
P'_i(u) &= s_{i,i}(u+e_{1-l, 1-l} + i/2 )   \dots s_{i,i}(u+e_{-2,-2} + i/2) \\
        &\qquad \times  s_{i,i}(u+e_{0,0} + i/2)
            s_{i,i}(u+e_{2,2}+i/2) \dots s_{i,i}(u+e_{l-1, l-1} + i/2), \notag
\end{align}
\begin{align*}
   P''_i(u) &=  s_{i,i}(u+e_{1-l, 1-l} + i/2 )   \dots s_{i,i}(u+e_{-2,-2} + i/2) \\
        &\qquad \times  s_{i,i}(u+e_{2,2}+i/2) \dots s_{i,i}(u+e_{l-1, l-1} + i/2),
\end{align*}
and
\begin{align*}
Q'_i(u) &= s_{i,i}(u+e_{1-l, 1-l} + i/2 )   \dots s_{i,i}(u+e_{-2,-2} + i/2) \\
        &\qquad \times  s_{i,i}(e_{0,0} + i/2)
            s_{i,i}(u+e_{2,2}+i/2) \dots s_{i,i}(u+e_{l-1, l-1} + i/2).
\end{align*}
Also observe that
\[
  P''_i(u) + \frac{1}{u} Q'_i(u) = \frac{1}{u} P_i'(u).
\]
So 
\begin{align*}
  P_i(u) - \frac{\eps}{2 u} Q_i(u) &= 
   P'_i(u) - \frac{\eps}{2}P''_i(u) - \frac{\eps}{2u}Q'_i(u) + \frac{1}{4u} P''_i(u)  \\
    &= P'_i(u) - \frac{\eps}{2 u} P'_i(u) + \frac{1}{4u} P''_i(u) \\
    &= P'_i(u)  - \frac{\eps}{2 u} P_i(u).
\end{align*}
Thus
\begin{align*}
S_{-\eta} & \circ \pi_{-\gamma}(s_{i,i}(\omega(u)) =P_i(u)  +
      \sum_{r=1}^\infty (-2 \eps u)^{-r} Q_i(u) \\
      &=  P_i(u)  -  
        \frac{\eps}{2u}  Q_i(u) + 
        \sum_{r=2}^\infty (-2 \eps u)^{-r}  Q_i(u) \\
       &= P_i'(u) - \frac{\eps}{2 u } P_i(u) -\frac{\eps}{2u} 
          \sum_{r=1}^\infty (-2 \eps u)^{-r}  Q_i(u)  \\
       &=  P_i'(u) - \frac{\eps}{2 u }  \left( P_i(u) +  
            \sum_{r=1}^\infty (-2 \eps u)^{-r}  Q_i(u)
           \right).
\end{align*}
So we have that 
\[
  P_i(u)  + \sum_{r=1}^\infty (-2 \eps u)^{-r} Q_i(u) =  
       P_i'(u) - \frac{\eps}{2 u }  \left( P_i(u) +  
            \sum_{r=1}^\infty (-2 \eps u)^{-r}  Q_i(u)
           \right),
\]
and solving this equation for 
\[
   P_i(u)  + \sum_{r=1}^\infty (-2 \eps u)^{-r} Q_i(u)
\]
gives that
\begin{equation} \label{bgklodd}
S_{-\eta} \circ \pi_{-\gamma}(s_{i,i}(\omega(u)) 
   = \sum_{r=0}^\infty (-2 \eps u)^{-r}  P'_i(u),
\end{equation}
which implies the lemma.
\end{proof}

Now we explain how irreducible highest weight $U(\mf{g},e)$-modules
under the BGK highest weight theory are related to the 
irreducible highest weight $U(\mf{g},e)$-modules from
Theorem \ref{t_class}.
To each 
skew-symmetric 
$n \times l$ tableaux 
we associate an
element of $\mf{t}^*$ 
in the following way.
For each $A = (a_{i,j})_{i \in \mathcal{I}_n,j \in \mathcal{I}_l}
\in \Tab_{n,l}$
 we define the weight
\[
  \lambda_A = \sum_{b \in \mathcal{I}_{nl} \cap \Z_{>0}} a_{\row(b),\row(b)} \eps_b  \in
   \mf{t}^*.
\]
Under this association, $\mf{t}^* = \Tab_{n,l}$, and
$\mf{t}^* / W_0 = \Row_{n,l}$.
Let $\Lambda_A$ denote the 
one-dimensional $U(\mf{g}_0, e)$-module
obtained by lifting the one-dimensional $S(\mf{t})^{W_0}$-module
corresponding to $\lambda_A$ through $\xi_{-\eta}$.

\begin{Theorem}
  Let $A \in \Row_{n,l}$.  Then $L(A) \cong L(\Lambda_A,\mf{q})$.
\end{Theorem}
\begin{proof}
First note that $L(A)$ is a BGK-highest weight module, since if
$v$ is a highest weight vector for $L(A)$ then
$s_{i,j}(\omega_r) v = 0$ when $i < j$ and 
$s_{i,i} (u^{-l} \omega(u)) v = \mu_i(u)$  for some 
$\mu_i(u) \in 1  + u^{-1} \C[[u^{-1}]]$.
Next by conferring with the definition of 
$L(A)$ in $\S1$, \eqref{seriesgens}, \eqref{bgkleven}, \eqref{pprime}, 
and \eqref{bgklodd} we see that the action of 
$s_{i,i} (u^{-l} \omega(u))$ on the highest weight vector for $L(A)$ and on the the highest weight
vector for $L(\Lambda_A, \mf{q})$ are the same.
Thus the theorem follows from \cite[Theorem 4.5]{BGK}.
\end{proof}

\section {Action of the Component Group $C$} \label{section6}
In this section we show how to explicitly calculate the action of 
the component group $C= C_G(e)/C_g(e)^\circ = 
 C_G(e,h,f) / C_G(e,h,f)^ \circ$  on 
the set of finite dimensional irreducible $U(\mf{g},e)$-modules.
Here $C_G(e,h,f)$ denotes the centralizer of the $\mf{sl}_2$-triple $(e,h,f)$
in the adjoint group $G$ of $\mf{g}$.
Recall Losev's near classification of finite dimensional irreducible 
representations of $U(\mf{g},e)$ from the introduction:
there exists a surjective map
\[
  \dagger : \Prim_{\operatorname{fin}} U(\mf{g},e) \twoheadrightarrow
  \Prim_{\overline {G.e}} U(\mf{g}),
\]
and the fibers of this map are precisely 
$C$-orbits.

In our special cases we can calculate explicitly the action of 
$C$ on the set of 
finite dimensional irreducible $U(\mf{g},e)$-modules, and therefore on
$\Prim_{\operatorname{fin}} U(\mf{g},e)$.
By \cite[Chapter 13]{C} the only cases where $C$ is not trivial
are the cases when $\eps = -$, and $n$ and $l$ are both even;
so unless otherwise indicated {\em we assume this for the rest of this section}.

Recall that in $\S$\ref{section2.1} we introduced an $n \times l$ rectangular array to specify coordinates.
Now we claim that
\[
  c = \sum_{\substack{ a \in \mathcal{I}_{nl} \\
                       \row(a) \notin \{\pm 1\}}} e_{a,a}
     + \sum_{\substack{ a,b \in \mathcal{I}_{nl} \\
               \col(a) = \col(b) \\
               \row(a) = 1 \\
               \row(b) = -1 }}
        e_{a,b} + e_{b,a}
\]
generates $C$.
To see this note
 that conjugating with $c$ simply transposes
each pair of indices $a,b \in \mathcal{I}_{nl}$ where 
$\col(a) = \col(b), \row(a) = 1, \row(b) = -1$.
Since this is an even number of transpositions, we have that
$\det c = 1$.
It is also clear that $c.J^-.c = J^-$ (recall that $\mf{g}$ is defined with
$J^-$ in \eqref{eq_J}) since for 
$a \in \mathcal{I}_{nl} \cap \Z_{>0}$ we have that
$c.e_{-a,a}.c = e_{-b,b}$ and
$c.e_{a,-a}.c = e_{b,-b}$ for some $b \in \mathcal{I}_{nl} \cap \Z_{>0}$.
Thus we have that  $c \in G$.
Furthermore, $c.h.c = h$ (see \eqref{eq_h} for the definition of $h$) 
since for $a \in \mathcal{I}_{nl}$
$c.e_{a,a}.c = e_{b,b}$ for some $b$ such that $\col(b) = \col(a)$.
Next note $c.e.c = e$ (see \eqref{eq_e} for the definition of $e$) 
since for $a,b \in \mathcal{I}_{nl}$ such that
$\row(a) = \row(b), \col(a) + 2 = \col(b)$, $c.f_{a,b}.c =f_{a,b}$
if $\row(a) \notin \{\pm 1\}$, and if $\row(a) = 1$ and $\col(b) \ge 1$ then
$c.f_{a,b}.c = f_{a',b'}$ where
$\col(a') = \col(a), \row(a') = -1$ and
$\col(b') = \col(b), \row(b') = -1$.
So we have that $c \in C_G(e,h,f)$.
Next we show that $c \notin C_G(e,h,f)^\circ$.
By \cite[$\S$3.8]{J} we have that $C_G(e,h,f) \cong O_n(\C)$.
Next observe that $C_G(h) \cong GL_n(\C)^{\times l/2}$ 
(confer \eqref{h}), and that the projection
of $c$ into any of these copies of $GL_n(\C)$ has determinant -1,
thus  $c \notin C_G(e,h,f)^\circ$.
Therefore $C= \langle c \rangle$.

To understand the action of $C$ on the set of finite dimensional irreducible 
$U(\mf{g},e)$-modules,
we calculate the action of $C$ on
$\set{\pr s_{i,j}(\omega(u)) }{ i,j \in \mathcal{I}_n}$.
Recall the definition of $s_{i,j}$ from \eqref{sij}.
Note that $c.s_{i,j}(e_{p,p}) = s_{i',j'}(e_{p,p})$ where
$i' = i$ if $i \notin \{\pm 1\}$, $i'=-i$ otherwise.
Thus $c.\pr s_{1,1}(\omega(u)) = \pr s_{-1,-1}(\omega(u))$, 
and $c.\pr s_{i,i}(\omega(u)) = \pr s_{i,i}(\omega(u))$ for $i \notin \{\pm 1\}$.
Since by Theorem \ref{gens} $\kappa_l(S_{i,j}(u) = \mu (s_{i,j}(\omega(u)))$,
we see that 
the action of $c$ on the
$U(\mf{g},e)$-module $L(A)$
is the same as the action of $\psi$.

We 
can now prove
Theorem \ref{t_c}:
\begin{proof}
We have that $L(A) = L(A^+) = L(\bar \mu(u))$ as $Y_n^+$-modules where 
$\mu_i(u) = (1+\frac{1}{2} u^{-1})^{-1} (1 + c_{i, -l} u^{-1}) (1+c_{i,2-l}u^{-1}) \dots (1+c_{i, l} u^{-1})$ are given from \eqref{lofa}. 
Since $\mu_i(u)$ must be a polynomial of degree at most $k$,
we must also have 
for each $i \in \{1, \dots, n-1\}$ that $c_{i,k} = \frac{1}{2}$ for some $k$.
After re-indexing we may assume that $c_{1,0}$ is the $\sharp'$-special 
element of $(c_{1,-l}, \dots, c_{1,l})$.
By Theorem \ref{th_sharp}  $L(\bar \mu_(u))^\sharp = 
L((\mu_1^\sharp(u), \mu_2(u), \dots, \mu_{n-1}(u)))$,
where  $\mu_1^\sharp(u) = (1+\frac{1}{2} u^{-1})^{-1} (1 + c_{1, -l} u^{-1}) 
\dots 
(1+(1-c_{1,0}) u^{-1}) \dots (1+c_{1, l} u^{-1})$.
So the skew symmetric $n \times l$ tableaux
$B = (b_{i,j})$ associated to 
$L((\mu_1^\sharp(u), \mu_2(u), \dots, \mu_{n-1}(u)))$ by \eqref{lofa} satisfies
$b_{1,0}  = -1/2 + (1-c_{1,0}) = -1/2 + (1-(a_{1,0}+1/2)) = -a_{1,0}$, and $b_{i,j} = a_{i,j}$
for all $(i,j) \ne (\pm1, 0)$.
\end{proof}

Throughout this paper $G$ denotes the adjoint group associated to $\mf{g}$.
It will be useful in future work to
consider the action of
the group $C' = C_{O_{n l}(\C)}(e,h,f) /  C_{O_{n l}(\C)}(e,h,f) ^ \circ$ 
on the set of finite dimensional irreducible $U(\mf{g},e)$-modules
in the case when
$\eps=+$ and $n$ is even and $l$ is odd.
In these cases, $C' \cong \Z_2$ and is generated by $c$ where
\[
  c = \sum_{\substack{ a \in \mathcal{I}_{nl} \\
                       \row(a) \notin \{\pm 1\}}} e_{a,a}
     + \sum_{\substack{ a,b \in \mathcal{I}_{nl} \\
               \col(a) = \col(b) \\
               \row(a) = 1 \\
               \row(b) = -1 }}
        e_{a,b} + e_{b,a}.
\]
As before, the action of $c$ on a finite dimensional
$U(\mf{g},e)$-module $L(A)$
is the same as the action of the $Y_n^+$-automorphism $\psi$, and so we obtain the following
theorem, whose proof is essentially the same as the proof for 
Theorem \ref{t_c}.
\begin{Theorem}
   Suppose that $n$ is even $l$ is odd integers and
   $\eps = +$.
   Let $A = \Std_{n,l}^+$,
  and let
$L(A)$ denote the corresponding finite
dimensional irreducible representation of $U(\mf{g},e)$.
Then the $\sharp$-special element of
$(a_{-1, l-1}, a_{-1, l-3}, \dots, a_{-1, l-1})$
is defined; let
$a$ denote the $\sharp$-special element of this list.
Let $c$ denote the generator of $C'$.
Then
$c.L(A) = L(B)$ where $B \in \Std_{n,l}^+$ has the same rows as $A$, except with one occurrence of $a$ replaced with $-a$ in row $-1$, and one occurrence of $-a$ replaced with $a$
in row $1$.
\end{Theorem}



\begin{thebibliography}{MNO }

\bibitem[B1]{Br}
J.Brown,
Twisted Yangians and finite $W$-algebras,
{\em Transform. Groups} {\bf 14}  (2009), 87--114,
{\tt arXiv:0710.2918}.

\bibitem[BB]{BB}
J. Brown and J. Brundan,
Elementary invariants for centralizers of nilpotent matrices,
 {\em J. Austral. Math. Soc.} {\bf 86} (2009), 1--15,
{\tt math.RA/0611024}.


\bibitem[BruG]{BG}
J. Brundan and S. Goodwin,
Good grading polytopes,
{\it Proc. London Math. Soc.} {\bf 94} (2007), 155--180,
{\tt math.QA/0510205}.

\bibitem[BroG]{BrG}
J. Brown and S. Goodwin,
Representation Theory of symplectic and orthogonal finite $W$-algebras
corresponding to even multiplicity nilpotent elements,
in preparation.

\bibitem[BGK]{BGK}
J. Brundan, S. Goodwin and A. Kleshchev,
 Highest weight theory for finite $W$-algebras,
{\it Int. Math. Res. Notices} (2008),
{\tt arXiv:0801.1337v1}.

\bibitem[BK1]{BKshifted}
J. Brundan and A. Kleshchev,
 Shifted Yangians and finite $W$-algebras,
{\it Adv. Math.} {\bf 200} (2006), 136--195, 
{\tt math.QA/0407012}.

\bibitem[BK2]{BKrep}
J. Brundan and A. Kleshchev,
 Representations of
shifted Yangians and finite $W$-algebras,
{\it Mem. Amer. Math. Soc.} {\bf 196} no. 918 (2008),
{\tt  arXiv:math.RT/0508003}.

\bibitem[BV]{BV}
D. Barbash and D.Vogan,
Primitive ideals and orbital integrals in complex classical groups,
{\em Math. Ann.} {\bf 259} (1982), 153--199.

\bibitem[C]{C}
R. W. Carter,
{\it Finite groups of Lie type: conjugacy classes and complex characters},
John Wiley \& Sons Ltd.,  1985.

\bibitem[DK]{DK}
A. De Sole and V. Kac,
Finite vs affine $W$-algebras,
{\em Jpn. J. Math.} {\bf 1} (2006), 137--261,
{\tt math-ph/0511055}.

\bibitem[D$^3$HK]{DDDHK}
A. D'Andrea, C. De Concini, A. De Sole, R. Heluani and V. Kac,
Three equivalent definitions of finite $W$-algebras,
appendix to \cite{DK}.

\bibitem[Dr]{Dr}
V. G. Drinfeld,
A new realization of Yangians and quantized affine algebras,
{\em Soviet Math. Dokl.} {\bf 36} (1988), 212--216.

\bibitem[EK]{EK}
P. Elashvili and V. Kac,
 Classification of good gradings of simple Lie algebras,
{\it Amer. Math. Soc. Transl.} {\bf 213}
(2005), 85--104,  {\tt math-ph/0312030}.

\bibitem[GG]{GG}
W. L. Gan and V. Ginzburg,
 Quantization of Slodowy slices,
{\it Int. Math. Res. Notices} {\bf 5} (2002) 243--255,
{\tt math.RT/0105225}.

\bibitem[Gi]{Gi}
V. Ginzburg,
Harish-Chandra bimodules for quantized Slodowy slices,
{\tt arXiv:0807.0339}.

\bibitem[Go]{Go}
S. Goodwin,
Translation for finite $W$-algebras,
{\tt arXiv:0908.2739v1}.

\bibitem[GRU]{GRU}
S. Goodwin, G. Roehrle, and G. Ubly,
On 1-dimensional representations of finite $W$-algebras associated to simple Lie algebras of exceptional type, 
{\tt arXiv:0905.3714v2}


\bibitem[J]{J}
J. C. Jantzen,
{\it Nilpotent orbits in representation theory},
Prog. Math {\bf 228} (2004).

\bibitem[K]{Ko}
B. Kostant,
On Whittaker vectors and representation theory,
{\em Inventiones Math.} {\bf 48} (1978), 101--184.

\bibitem[KRW]{KRW}
V. Kac, S. Roan, and M. Wakimoto,
Quantum Reduction for Affine Superalgebras,
{\em Commun. Math. Phys.} {\bf 241} (2003), 307--342,
{\tt math-ph/0302015}.

\bibitem[Lo1]{Lo1}
I. Losev,
Quantized symplectic actions and $W$-algebras, {\tt math.RT/0707.3108}.

\bibitem[Lo2]{Lo2}
I. Losev,
Finite dimensional representations of $W$-algebras, {\tt arXiv:0807.1023v1}.

\bibitem[Lo3]{Lo3}
I. Losev,
On the structure of the category $\mathcal{O}$ for $W$-algebras, {\tt arXiv:0812.1584}.

\bibitem[Lo4]{Lo4}
I.~Losev,
1-dimensional representations and parabolic induction for $W$-algebras,
preprint,
{\tt arXiv:0906.0157}.


\bibitem[Ly]{Ly}
T. E. Lynch,
{\it Generalized Whittaker vectors and representation theory},
Ph.D. Thesis, 
MIT, Cambridge, MA, 1979.

\bibitem[M]{M}
A. Molev, 
 Finite-dimensional irreducible representations of twisted Yangians,
{\it J. Math. Phys.} 
{\bf 39} (1998), 5559-5600,
{\tt arXiv:q-alg/9711022}.

\bibitem[MNO]{MNO}
A. Molev, M. Nazarov and G. Olshanskii,
{\it Yangians and classical Lie algebras},
Russian Math. Surveys {\bf 51} (1996), 205--282.

\bibitem[P1]{P1}
A. Premet,
 Special transverse slices and their enveloping algebras,
{\it Adv. Math.} {\bf 170} (2002), 1--55.

\bibitem[P2]{P2}
A. Premet,
 Enveloping algebras of Slodowy slices and the Joseph ideal,
{\it J.\ Eur.\ Math.\ Soc.} {\bf 9} (2007), 487--543,
{\tt math.RT/0504343}.

\bibitem[P3]{P3}
A. Premet,
Commutative quotients of finite $W$-algebras,
{\tt arXiv:0809.0663}.

\bibitem[PPY]{PPY}
D. Panyushev, A. Premet, and O. Yakimova,
On symmetric invariants of centralisers in reductive Lie algebras,
{\it J. Algebra} {\bf 313} (2007), 343--391,
{\tt math.RT/0610049}.


\bibitem[R]{R}
E. Ragoucy,
 Twisted Yangians and folded $W$-algebras,
{\it Int. J. Mod. Phys. A} {\bf 16} 13 (2001), 2411--2433,
{\tt math.QA/0012182}.

\bibitem[Sk]{Skryabin}
S. Skryabin,
A category equivalence,
Appendix to \cite{P1}.

\end{thebibliography}
\end{document}